\documentclass[review,1p]{elsarticle}
\usepackage{hyperref,xspace}

\usepackage[english]{babel}%

\usepackage{amsmath,amsfonts,amssymb,amsthm,mathrsfs,amsbsy,dsfont}
\usepackage{textcomp}
\usepackage{subfigure,graphics,pdfpages,tikz}
\usepackage{multirow}
\usepackage{bbold}
\usepackage{xspace}
\usepackage{bm}
\usepackage{color}\newcommand{\code}[1]{\texttt{#1}}
\usepackage{geometry}
\usepackage{lscape}
\usepackage{algorithm,algorithmic, lscape}

\usetikzlibrary{decorations}
\usetikzlibrary{decorations.pathmorphing}
\usetikzlibrary{decorations.pathreplacing}
\usetikzlibrary{decorations.shapes}
\usetikzlibrary{decorations.text}
\usetikzlibrary{decorations.markings}
\usetikzlibrary{decorations.fractals}
\usetikzlibrary{decorations.footprints}

\newtheorem{lem}{Lemme} 
 
\newtheorem{rk}{Remark}

\newcommand{\Dmax}{D_{\max}}
\renewcommand{\H}{\mathcal{H}}

\newcommand{\1}{\mathds{1}}
\newcommand{\R}{\mathbb{R}}
\newcommand{\X}{\mathcal{X}}
\newcommand{\K}{\mathbf{K}}
\newcommand{\T}{\mathcal{T}}
\newcommand{\Tn}{\mathcal{T}_n}

\newcommand{\Z}{\mathbf{Z}}
\renewcommand{\O}{\mathcal{O}}
\newcommand{\E}{\mathds{E}}
\newcommand{\argmin}{\operatornamewithlimits{arg\ min}}
\newcommand{\paren}[1]{\left( #1 \right)}
\newcommand{\abs}[1]{\left\lvert #1 \right\rvert} 
\newcommand{\croch}[1]{\left[\, #1 \,\right]}
\newcommand{\acc}[1]{\left\{ #1 \right\}}
\newcommand{\norm}[1]{\left\| #1 \right\|}
\newcommand{\normH}[1]{\left\Vert #1 \right\Vert_{\H}}
\newcommand{\normHn}[1]{\left\Vert #1 \right\Vert_{\H,n}		}
\newcommand{\scal}[1]{\left \langle #1 \right\rangle}
\newcommand{\scalH}[1]{\left \langle #1 \right\rangle_{\H}}
\newcommand{\tauh}{\hat{\tau}}
\newcommand{\muh}{\hat{\mu}^\tau}
\newcommand{\muhD}{\hat{\mu}_{\hat{\tau}_D}}
\newcommand{\for}{\textbf{for} }

\usepackage[normalem]{ulem}

\newcommand{\Card}[1]{\mathrm{Card}\paren{ #1 }}

\newcommand{\Kern}{\emph{Kernseg}\xspace}
\newcommand{\aKern}{\emph{ApKS}\xspace}
\newcommand{\ECP}{ECP\xspace}
\newcommand{\RBS}{RBS\xspace}
\newcommand{\KernSegLin}{KS.Lin\xspace}
\newcommand{\KernSegGauss}{KS.Gau\xspace}
\newcommand{\KernSegECP}{KS.ECP\xspace}

\begin{document}
\begin{frontmatter}

\title{New efficient algorithms for multiple change-point detection with kernels}

\author[INRIA,UL1]{A. C\'elisse}
\author[UL2,INRIA]{G. Marot}
\author[UL2,SG]{M. Pierre-Jean}
\author[SG,IPS2]{G.J. Rigaill}
\address[UL2]{\textbf{Univ. Lille Droit et Sant\'e}
EA 2694 - CERIM, F-59000 Lille, France}
\address[SG]{\textbf{UMR 8071 CNRS - Universit\'{e} d'Evry - INRA}
Laboratoire Statistique et G\'{e}nome
Evry}
\address[INRIA]{\textbf{Inria Lille Nord Europe}
\'Equipe-projet Inria \textsc{MODAL} }
\address[IPS2]{\textbf{Institute of Plant Sciences Paris-Saclay}, UMR 9213/UMR1403, CNRS, INRA, Universit{\'e} Paris-Sud, Universit{\'e} d'Evry,
Universit{\'e} Paris-Diderot, Sorbonne Paris-Cit{\'e}
}
\address[UL1]{\textbf{Univ. Lille Sciences et Technologies, CNRS,} UMR 8524 - Laboratoire Paul Painlev\'e, F-59000 Lille, France}

\begin{abstract}
Several statistical approaches based on reproducing kernels have been proposed to
detect abrupt changes arising in the full distribution of the observations and not only in the mean or variance.
Some of these approaches enjoy good statistical properties (oracle inequality, \ldots).
Nonetheless, they have a high computational cost both in terms of time and memory. 
This makes their application difficult even for small and medium sample sizes ($n< 10^4$).
This computational issue is addressed by first describing a new efficient and exact algorithm for kernel multiple change-point detection with an improved worst-case complexity that is quadratic in time and linear in space. It allows dealing with medium size signals (up to $n \approx 10^5$). 
Second, a faster but approximation algorithm is described. It is based on a low-rank approximation to the Gram matrix. 
It is linear in time and space. This approximation algorithm can be applied to large-scale signals ($n \geq 10^6$).
These exact and approximation algorithms have been implemented in \texttt{R} and \texttt{C} for various kernels.
The computational and statistical performances of these new algorithms have been assessed through empirical experiments.
The runtime of the new algorithms is observed to be faster than that of other considered procedures. 

Finally, simulations confirmed the higher statistical accuracy of kernel-based approaches
to detect changes that are not only in the mean. These simulations also illustrate the flexibility of kernel-based approaches to analyze complex biological profiles made of DNA copy number and allele B frequencies. 

%
{An R package implementing the approach will be made available on github.}

\end{abstract}


\end{frontmatter}

\section{Introduction}\label{sec:background}


In this paper we consider the multiple change-point detection problem \citep{BrodskyDarkhovsky_2013} where the goal is to recover abrupt changes arising in the distribution of a sequence of $n$ independent random variables $X_1,\ldots,X_n$ observed at respective time $t_1<t_2<\ldots < t_n$.


\paragraph{State-of-the-art}
Many parametric models (Normal, Poisson,\ldots) have been proposed \citep{HKYAWK_2003,rigaill2012exact,Cleynen_Lebarbier:2014}. These models allow detecting different types of changes: in the mean, in the variance and in both the mean and variance  (see also \citep{HKYAWK_2003,JMVYWM_2003,picard05a-statistical}).
Efficient algorithms and heuristics have been proposed for these models. Some of them scale in $\mathcal{O}(n\log(n))$ or even in $\mathcal{O}(n)$.
In practice, these parametric approaches have proven to be successful for various application fields (see for example \cite{hocking13learning,cleynen2014comparing}). 
However one of their main drawbacks is their lack of flexibility. For instance, any change of distributional assumption requires the development of a new dedicated inference scheme.

By contrast, the recently proposed kernel change-point detection approach \cite{harchaoui2007retrospective,Arlot2011} is more generic.
It has the potential to detect any change arising in the distribution, which is not easily captured by standard parametric models.
More precisely in this approach, the observations are first mapped into a Reproducing Kernel Hilbert Space (RKHS) through a kernel function \citep{Aronszajn:1950}. 
The difficult problem of detecting changes in the distribution is then recast as simply detecting changes in the mean element of observations in the RKHS, which is made possible using the well-known kernel trick.

One practical limitation of this kernel-based approach is its considerable computational cost owing to the use of a $n\times n$ Gram matrix combined with a dynamic programming algorithm \citep{Auger_Lawrence:1989}. 
More precisely \cite{harchaoui2007retrospective} described a dynamic programming algorithm to recover the best segmentation from $1$ to $\Dmax$ segments. They claim that their algorithm has a $\mathcal{O}(\Dmax n^2)$ time complexity.
However, the latter is not described in full details and its straightforward implementation is not efficient.
First, it requires the storage of a $n \times n$ cost matrix (personal communication with the first author of \cite{harchaoui2007retrospective} who was kind enough to send us his code).
Thus the algorithm has a $\mathcal{O}(n^2)$ space complexity, which is a severe limitation with nowadays sample sizes. For instance analyzing a signal of length $n=10^5$ requires storing a $10^5 \times 10^5$ matrix of doubles, which takes $80$ GB.
Second, computing the cost matrix is not straightforward. In fact simply using formula (8) of \cite{harchaoui2007retrospective} to compute each term of this cost matrix leads to an $\mathcal{O}(n^4)$ time complexity.

\paragraph{Contributions}

The present paper contains several contributions to the computational aspects and the statistical performance of the kernel change-point procedure introduced by \cite{Arlot2011}.

The first one is to describe a new algorithm to simultaneously perform the dynamic programming step of \cite{harchaoui2007retrospective} and also compute the required elements of the cost matrix on the fly.
On the one hand, this algorithm has a complexity of order $\mathcal{O}(\Dmax n^2)$ in time and $\mathcal{O}(\Dmax n)$ in space (including both the dynamic programming and the cost matrix computation). We also emphasize that this improved space complexity comes without an increased time complexity. 
This is a great algorithmic improvement upon the change-point detection approach described by \cite{Arlot2011} since it allows the efficient analysis of signals with up to $n=10^5$ data-points in a matter of a few minutes on a standard laptop.

On the other hand, our approach is generic in the sense that it works for any positive semidefinite kernels. 
Importantly one cannot expect to exactly recover the best segmentations from $1$ to $\Dmax$ segments in less than $\mathcal{O}(\Dmax n^2)$ without additional specific assumptions on the kernel. Indeed, computing the cost of a given segmentation has already a time complexity of order $\mathcal{O}(n^2)$.
%

It is also noticeable that our algorithm can be applied to other existing strategies such as the so-called \ECP \citep{matteson2014nonparametric}.
To be specific, we show that the \emph{divisive clustering algorithm} it is based on and that provides an approximate solution with a complexity of order $O(n^2)$ in time and space can be replaced by our algorithm that provides the exact solution with the same time complexity and reduced memory complexity.

Our second contribution is a new algorithm dealing with larger signals ($n>10^5$) based on a low-rank approximation to the Gram matrix.
This computational improvement is possible at the price of an approximation. It returns approximate best segmentations from 1 to $\Dmax$ segments with a complexity of order $\mathcal{O}(\Dmax p^2n)$ in time and $\mathcal{O}((\Dmax+p)n)$ in space, where $p$ is the rank of the approximation.

The last contribution of the paper is the empirical assessment of the statistical performance of the KCP procedure introduced by \cite{Arlot2011}.
This empirical analysis is carried out in the biological context of detecting abrupt changes from a two-dimensional signal made of DNA copy numbers and allele B fractions \citep{lai12change-point}. The assessment is done by comparing our approach to state-of-the-art alternatives on resampled real DNA copy number data \cite{Pierre-Jean08092014,matteson2014nonparametric}.
This illustrates the versatility of the kernel-based approach. To be specific this approach allows the detection of changes in the distribution of such complex signals without explicitly modeling the type of change we are looking for.


\bigskip

The remainder of the paper is organized as follows.
In Section~\ref{sec:kernel-framework}, we describe our kernel-based framework and detail the connection between detecting abrupt changes in the distribution and model selection as described in \cite{Arlot2011}. 
A slight generalization of the KCP procedure \citep{Arlot2011} is also derived in Section~\ref{sec.constrained.segmentations} by introducing a new parameter $\ell$ encoding an additional constraint on the minimal length of any candidate segment. This turns out to be particularly useful in low signal-to-noise ratio settings.
The versatility of this kernel-based framework is emphasized in Section~\ref{sec.link.other.approaches} where it is shown how the \ECP approach \citep{matteson2014nonparametric} can be rephrased in terms of kernels.
Our main algorithmic improvements are detailed and justified in Section~\ref{sec:algorithm}. 
We empirically illustrate the improved runtime of our algorithm and compare it to the ones of \ECP and \RBS in Section~\ref{sec.runtimes.comparison}.
In Section~\ref{sec:heuristic} we detail our faster (but approximate) algorithm used to analyze larger profiles ($n>10^5$). It is based on the combination of a low-rank approximation to the Gram matrix and the binary segmentation heuristic \citep{yang12simple}.
An empirical comparison of the runtimes of the exact and approximate algorithms is provided in Section~\ref{sec.runtimes.approx.solution}.
Finally, Section~\ref{sec:eval-segm} illustrates the statistical performance of our kernel-based change-point procedure in comparison with state-of-the-art alternatives in the context of biological signals such as DNA copy numbers and allele B fractions  \citep{lai12change-point}.

\section{Kernel framework}
\label{sec:kernel-framework}

In this section we recall the framework of \cite{harchaoui2007retrospective} where detecting changes in the distribution of a complex signal is rephrased as detecting changes of the mean element of a sequence of points in a Hilbert space. 
Then we detail the so-called KCP (Kernel Change-point Procedure) \citep{Arlot2011}, which has been proved to be optimal in terms of an oracle inequality.

\subsection{Notation}\label{sec:notations}
Let $X_1,X_2,\ldots,X_n \in\mathcal{X}$ be a time-series of $n$ independent random variables, where $\mathcal{X}$ denotes any set assumed to be \emph{separable} \citep{Die_Bac:2014} throughout the paper.
Let $k:\ \mathcal{X} \times \mathcal{X} \to \R$ denote a symmetric positive semi-definite kernel \citep{Aronszajn:1950}, and $\mathcal{H}$ be the associated reproducing kernel Hilbert space (RKHS). We refer to \cite{Ber_Tho:2004} for an extensive presentation about kernels and RKHS.
Let us also introduce the canonical feature map $\Phi :\ \mathcal{X} \to \mathcal{H}$ defined by $\Phi(x) = k(x,\cdot) \in \H $, for every $x\in\mathcal{X}$. This canonical feature map allows to define the inner product on $\H$ from the kernel $k$, by 
\begin{align}\label{eq.kernel.trick}
\forall x, y\in \mathcal{X},\qquad   \scalH{\Phi(x), \Phi(y) } = k(x,y) .
\end{align}

\paragraph{The asset of kernels} 
One main advantage of kernels is to enable dealing with complex data of any type provided a kernel can be defined. In particular no vector space structure is required on $\X$. For instance $\X$ can be a set of DNA sequences, a set of graphs or a set of distributions to name but a few examples (see \citep{Gart_2008} for various instances of $\X$ and related kernels).
Therefore, as long as a kernel $k$ can be defined on $\X$, any element $x\in\X$ is mapped, through the canonical feature map $\Phi$, to an element of the Hilbert space $\H$. This provides a unified way to deal with different types of (simple or complex) data.
Then for every index $1 \leq t\leq n$, let us note 
\begin{eqnarray}  \label{eq:repr-RKHS}
  Y_t = \Phi (X_t) \in \H .
\end{eqnarray}
From now on, we will only consider the following sequence $Y_1,\ldots,Y_n \in \H$ of independent Hilbert-valued random vectors.

\paragraph{The kernel trick} 
As a space of functions from $\X$ to $\R$, the RKHS $\H$ can be infinite dimensional. From a computational perspective one could be worried that manipulating such objects is computationally prohibitive. 
However this is not the case and our algorithm relies on the so-called \emph{kernel trick}, which consists in translating any inner product in $\H$ in terms of the kernel $k$ by use of Eq.~\eqref{eq.kernel.trick}. 
For every $1 \leq i,j \leq n$, it results
\begin{align*}
  \scalH{ Y_i, Y_j } = k(X_i,X_j) = \K_{i,j} ,
\end{align*}
 where $\K_{i,j}$ denotes the $(i,j)$-th coefficient of the $n\times n$ Gram matrix $\K=\croch{ k(X_i,X_j) }_{1\leq i,j\leq n}$.


\subsection{Detecting changes in the distribution using kernels}\label{sec:characteristic}

Let us consider the model introduced by \cite{Arlot2011}, which connects every $Y_t$ to its ``mean'' $\mu^{\star}_t \in \H$ by
\begin{eqnarray}\label{eq.model.RKHS}
\forall \ 1\leq t\leq n, \qquad Y_t = \Phi(X_t) = \mu^{\star}_t+ \epsilon_t \in \H ,
\end{eqnarray}
where $\mu^{\star}_t$ denotes the \emph{mean element} associated with the distribution $\mathds{P}_{X_t}$ of $X_t$, and $\epsilon_t = Y_t - \mu^{\star}_t $.
Let us also recall \citep{Led_Tal:1991} that if $\X$ is separable and $\E\croch{ k(X_t,X_t) }< +\infty$, then $\mu^{\star}_t$ exists and is defined as the unique element in $\H$ such that
\begin{eqnarray}  \label{eq:mean-element}
  \forall \ f \in \mathcal{H}, \quad \scalH{ \mu^{\star}_t, f }  =  \E \scalH{\Phi(X_t), f } . 
\end{eqnarray}

%
For characteristic kernels \citep{SriFuLa_2010}, a change in the distribution of $X_t$ implies a change in the mean element $\mu^{\star}_t$, that is
\begin{eqnarray}  \label{eq:diff-dist-imply-diff-ME}
\forall \ 1\leq i\neq j \leq n ,\qquad   \mathds{P}_{X_i}\neq \mathds{P}_{X_j} \Rightarrow \mu^{\star}_i \neq\mu^{\star}_j ,
\end{eqnarray}
the converse implication being true by definition of $\mu^{\star}_t$ in Eq.~\eqref{eq:mean-element}.
The idea behind kernel change-point detection \citep{Arlot2011} is to translate the problem of detecting changes in the distribution into detecting changes in the mean of Hilbert-valued vectors.
\begin{rk}
	When considering $\X \subset \R^q$ for some integer $q>0$, several classical kernels are characteristic. For instance,
	\begin{itemize}
		\item The Gaussian kernel: $k(x,y) = e^{-\norm{x-y}^2/\delta }$, with $x,y\in \R^q$ and $\delta>0$,
		\item The Laplace kernel: $k(x,y) = e^{- \norm{x-y}/\delta }$, with $x,y\in \R^q$ and $\delta>0$, 
		 \item The exponential kernel: $k(x,y) = e^{-\scal{x,y}_q/\delta}$, with $x,y\in \R^q$ and $\delta>0$,
	\end{itemize}
	where $\norm{\cdot}$ and $\scal{\cdot,\cdot}_q$ respectively denote the usual Euclidean norm and inner product in $\R^q$.
	The energy-based kernel discussed in Section~\ref{sec.link.other.approaches} is also a characteristic kernel (see Lemma~1 in \cite{Matteson_James:2014}).
	However, with more general sets $\X$, building a characteristic kernel is challenging as illustrated by \citep{SriFuLa_2010} and \cite{ChristSteinw_2010}.
\end{rk}
Let us also notice that the procedure developed by \cite{Arlot2011} can be seen as a ``kernelized version'' of the procedure proposed by \cite{Lebarbier2005}, which was originally designed to detect changes in the mean of real-valued variables.

\subsection{Statistical framework}\label{sec:statistical-model}

From Eq.~\eqref{eq:diff-dist-imply-diff-ME} it results that any sequence of abrupt changes in the distribution along the time corresponds to a sequence of $D^*$ true change-points $1=\tau^*_1< \tau^*_2 < \ldots < \tau^*_{D^*} \leq n $ (with $\tau^*_{D^*+1}  = n+1$ by convention) such that
\begin{align*}
  \mu^*_1 = \ldots = \mu^*_{\tau^*_1-1} \neq  \mu^*_{\tau^*_1}=\ldots = \mu^*_{\tau^*_2-1} \neq \ldots \neq \mu^*_{\tau^*_{D^*}} = \ldots = \mu^*_{n} .
\end{align*}
In other words we get that $\mu^* = (\mu^*_1,\ldots,\mu^*_n)^\prime \in\H^n$ is piecewise constant.

From a set of $D$ candidate change-points $ 1 = \tau_1< \ldots < \tau_{D}\leq n$, let $\tau$ be defined by 
\begin{align*} 
\tau = \paren{ \tau_1, \tau_2, \ldots ,\tau_{D-1}, \tau_{D} } ,
\end{align*}
with the convention $\tau_1=1$ and $\tau_{D+1}=n+1$.
With a slight abuse of notation, we also call $\tau$ the segmentation of $\acc{1,\ldots,n}$ associated with the change-points $ 1 = \tau_1< \ldots < \tau_{D}\leq n$.
The estimator $\muh=\paren{\muh_1,\ldots,\muh_n}^\prime\in\H^n$ of $\mu^*=\paren{\mu^*_1,\ldots,\mu^*_n}^\prime$ 
proposed by \cite{Arlot2011} is defined by
\begin{eqnarray*}
\forall \ 1\leq i \leq D,\quad \forall \ t\in\acc{\tau_i,\ldots,\tau_{i+1}-1},\quad  \muh_t = \frac{1}{\tau_{i+1}-\tau_i} \displaystyle \sum_{t^\prime =\tau_i}^{\tau_{i+1}-1} Y_{t^\prime} .
\end{eqnarray*} 
The performance of $\muh$ is measured by the quadratic risk: 
\begin{align*}
\mathcal{R}( \muh) = \E\croch{ \norm{ \mu^* - \muh }^2_{\H,n}} = \E\croch{ \sum_{i=1}^n \normH{ \mu^*_i - \muh_i }^2 } ,
\end{align*}
where $\normH{\cdot}$ denotes the norm in the Hilbert space $\H$.

\subsection{Model selection}\label{sec:model-select}

If the signal-to-noise ratio is small, \cite{Arlot2011} emphasized that all true change-points cannot be recovered without including false change-points.
This leads them to define the best segmentation $\tau^*$ (for a finite sample size) as
\begin{align*}
 \tau^* = \argmin_{ \tau \in \Tn} \norm{ \mu^* - \muh }_{\H,n} ,
\end{align*}
where $\Tn$ denotes the collection of all possible segmentations $\tau$ of $\acc{1, \ldots,n}$ with at most $D_{\max}$ segments.
When the signal-to-noise ratio is large enough, $\tau^*$ coincides with the true segmentation.

As a surrogate to the previous unknown criterion, \cite{Arlot2011} optimize the following penalized criterion
\begin{align} \label{def.penalized.criterion.global}
  \hat \tau = \argmin_{\tau\in\Tn} \acc{ \norm{ Y - \muh }^2_{\H,n} + \mathrm{pen}(\tau) } ,\qquad \mbox{with}\quad  \mathrm{pen}(\tau) = c_1 D_\tau  + c_2 \log{ n-1 \choose D-1 } ,
\end{align}
where $Y=(Y_1,\dots,Y_n)^\top\in\R^n$, $c_1,c_2>0$ are constants to be fixed, and $D_\tau$ denotes the number of segments of the segmentation $\tau$.
Since this penalty only depends on $\tau$ through $D_\tau$, optimizing \eqref{def.penalized.criterion.global} can be formulated as a two-step procedure.
The first step consists in solving:
\begin{align} \label{def.penalized.criterion.global.first.step}
  \forall \ 1\leq D \leq D_{\mathrm{max}},\quad \hat{\tau}_D = \argmin_{\tau \in \T_D} \norm{ Y - \muh }^2_{\H,n} ,
\end{align}
where $\T_D$ denotes the set of segmentations with $D$ segments. This optimization problem, which is usually solved by dynamic programming \cite{Auger_Lawrence:1989,rigaill2012exact}, is computationally hard since the cardinality of $\T_D$ is ${n-1\choose D-1}$.
The second step is to straightforwardly optimize:
\begin{align} \label{def.penalized.criterion.global.second.step}
  \hat D = \argmin_{1\leq D\leq D_{\max}} \acc{ \norm{ Y - \muhD }^2_{\H,n} + \mathrm{pen}(\tauh_D) } \qquad \mbox{and}\quad \tauh = \tauh_{\hat D} .
\end{align}

The right-most term in the penalty \eqref{def.penalized.criterion.global} accounts for the number of candidate segmentations with $D$ segments (see the comments of Theorem~2 in \cite{Arlot2011}). Intuitively this term balances the trend of the estimator \eqref{def.penalized.criterion.global.first.step} to overfit because of the large number of candidate segmentations.
\begin{rk}
		It is important to notice that the above two-step procedure depends on the hyper-parameter $D_{\max}$, which is the maximum number of segments of the candidate segmentations. 
		%
		
		In fact choosing an appropriate $D_{\max}$ is related to the calibration of constants $c_1$ and $c_2$ in the penalty term of \eqref{def.penalized.criterion.global}.
		Since the optimal values of $c_1$ and $c_2$ depend (at least) on the variance of the signal at hand, they have to be calibrated in a data-driven way. Here they have been calibrated by using the so-called \emph{slope heuristic} technique described in \cite{Arlot2011} (see also the numerical experiments in Section~\ref{sec:eval-segm} for more details).
		In particular $D_{\max}$ has to be chosen large enough to make the slope heuristic work well. 
		Given some prior knowledge of an adequate range of values for $D$ taking $D_{max}$ to be 10 to 20 times larger than that seems to work well in practice. Typically for copy number data (see Section~\ref{sec.DNA.compy.number.data}) one rarely expect more than 10 change-points per chromosome and taking $D_{max} \approx 100$ or $200$ often make sense.

	%
	\end{rk}

From a theoretical point of view, this model selection procedure has been proved to be optimal in terms of an oracle inequality by \cite{Arlot2011}. This is the usual non-asymptotic optimality result for model selection procedures \citep{Bir_Mas:2006}. This procedure has also been proved to provide consistent estimates of the change-points \citep{GarreauArlot_2016}.
However, from a computational point of view, the first step (\textit{i.e.} solving Eq.~\eqref{def.penalized.criterion.global.first.step}) remains challenging.
Indeed existing dynamic programming algorithms are time and space consuming when used in the kernel framework as it will be clarified in Section~\ref{sec.dynamic.programming.classical}.
The main purpose of the present paper is to provide a new computationally efficient algorithm to solve Eq.~\eqref{def.penalized.criterion.global.first.step}.
Our new algorithm has a reduced space and time complexity and allows the analysis of signals larger than $n= 10^4$.

\subsection{Low signal-to-noise and minimal length of a segment}\label{sec.constrained.segmentations}

In settings where the signal-to-noise ratio is weak (see for instance Figure~\ref{fig:example-data} where the tumor percentage is low) change-point detection procedures are more likely to put changes in noisy regions. This results in overfitting and meaningless small segments \cite{Arlot:Celisse:2011}.
A common solution is to include a constraint on the minimum length $\ell$ of segments. For instance by default \ECP enforces that the estimated segmentation has segments with at least $\ell = 30$ points \cite{ecp}.

One important side effect of this constraint on $\ell$ is that the total number of candidate segmentations with $D$ segments quickly decreases with $\ell$. 
Therefore the penalty in \eqref{def.penalized.criterion.global} has to be modified.

The following lemma gives the cardinality of this set of segmentations.
\begin{lem}\label{lem.combinatoire}
	Let $\Tn^\ell(D)$ denote the set of segmentations of $\paren{1,\ldots,n}$ in exactly $D\geq 1$ segments such that the length of each segment is at least $\ell \geq 1$.
	Then the cardinality of $\Tn^\ell(D)$ satisfies
	\begin{align*}
	\Card{ \Tn^\ell(D) } = { n-D(\ell-1) -1 \choose D-1} .
	\end{align*}
\end{lem}
Let us notice that if $\ell=1$, one recovers the usual cardinality that is used in the penalty (see Eq.~\eqref{def.penalized.criterion.global}).
As an illustration of the influence of the constraint on $\ell$, let us consider the set-up where $n=100$, $D=10$, and $\ell =10$. Then the size of the unconstrained set of segmentations with $10$ segments $\T_{100}(10) = \T_{100}^1(10)$ is $\Card{ \T_{100}(10) } \approx 1.7 \cdot 10^{12}$, whereas the constrained set $\T_{100}^{10}(10)$ is smaller since its cardinality is equal to 1.

	\begin{proof}[Proof of Lemma~\ref{lem.combinatoire}]
		The proof consists in showing that there is a one-to-one mapping between the set $\Tn^\ell(D)$ of segmentations of $\paren{1,\ldots,n}$ with $D$ segments of length at least $\ell\geq 1$, and the set $\mathcal{S}_n^\ell(D)$ of segmentations of $\paren{1,\ldots,n-D(\ell-1)}$ with $D$ (non-empty) segments.
		
		Let us consider one segmentation $\tau \in \Tn^\ell(D)$. Since each segment of $\tau$ is of length at least $\ell$, let us remove $\ell-1$ points from the left edge of each of the $D$ segments. Then the resulting segmentation belongs to $\mathcal{S}_n^\ell(D)$.
		
		Conversely, take one segmentation $\tau \in \mathcal{S}_n^\ell(D)$.
		Then each segment of $\tau$ contains at least one point. Adding $\ell-1$ points to each segment (from the left edge) clearly provides a segmentation with $D$ segments of length at least $\ell$. This allows to conclude.
	\end{proof}
	
	This leads to the following generalized change-points detection procedure involving a constraint on the minimum length $\ell\geq 1$ of each segment.
	\begin{itemize}
		\item[Step 1: ] Solve
		\begin{align} \label{def.penalized.criterion.global.first.step.constrained}
		\forall \ 1\leq D \leq D_{\mathrm{max}},\quad \hat{\tau}_D^\ell = \argmin_{\tau \in \T_n^\ell(D)} \norm{ Y - \muh }^2_{\H,n} ,
		\end{align}
		where $\T_n^\ell(D)$ denotes the set of segmentations with $D$ segments of length at least $\ell\geq 1$. 
		
		\item[Step 2: ] Find
		\begin{align} \label{def.penalized.criterion.global.second.step.constrained}
		\hat D^\ell = \argmin_{D} \acc{ \norm{ Y - \hat \mu^{\hat{\tau}_D^\ell } }^2_{\H,n} + \mathrm{pen}_\ell(\tauh_D^\ell) } \qquad \mbox{and}\quad \tauh^\ell = \tauh_{\hat D^\ell}^\ell ,
		\end{align}
		where $\mathrm{pen}_\ell(\tau) = c_1 D_{\tau} + c_2 \log{ n-D_\tau(\ell-1) -1 \choose D_\tau-1}  $.
	\end{itemize}

	Let us emphasize that this generalized procedure, including this additional 
parameter $\ell\geq 1$, is very similar the previous one.
The optimization step of Eq.~\eqref{def.penalized.criterion.global.first.step.constrained} is performed by dynamic programming up to a minor change of implementation. The optimization of the second step \eqref{def.penalized.criterion.global.second.step.constrained} remains unchanged except it involves a slightly different penalty shape. The tuning of the constant $c_1$ and $c_2$ is still made by the slope heuristic (see Section~\ref{sec:eval-segm}).

\subsection{A link between kernels and energy-based distances}\label{sec.link.other.approaches}

Note that the kernel-based framework developed in Sections~\ref{sec:notations}--\ref{sec:statistical-model} is very general. 
Various existing procedures can be rephrased in this framework by use of a particular kernel. For example the procedure of \citep{Lebarbier2005}, which is devoted to the detection of changes in the mean of a one-dimensional real-valued signal, reduces to ours by use of the linear kernel. 
More interestingly the procedure called \ECP developed by \citep{matteson2014nonparametric} and that relies on an energy-based distance to detect changes in multivariate distributions, can also be integrated into our framework using a particular kernel as explained in what follows.

For every $\alpha\in(0,2)$, let us define $\rho_{\alpha}(x,y) = \norm{x-y}^{\alpha} $, where $x,y\in\R^q$ and $\norm{\cdot}$ denotes the Euclidean norm on $\R^q$. Then $ \rho_{\alpha}$ is a semimetric of negative type \citep{Berg_Christ_Ressel:1984}, and for any independent random variables $X,X^\prime,Y,Y^\prime\in \R^d$ with respective probability distributions satisfying $P_X=P_{X^\prime}$ and $P_Y=P_{Y^\prime}$, \cite{matteson2014nonparametric} introduce the energy-based distance:
\begin{align}\label{eq.energy.distance}
  \mathcal{E}(X,Y;\alpha) = 2 E\croch{ \rho_{\alpha}(X,Y)} - E\croch{ \rho_{\alpha}(X,X^\prime)} - E\croch{ \rho_{\alpha}(Y,Y^\prime)} ,
\end{align}
with the assumption that $\max\paren{E\croch{ \rho_{\alpha}(X,X^\prime)},E\croch{ \rho_{\alpha}(X,Y)},E\croch{ \rho_{\alpha}(Y,Y^\prime)}}<+\infty$.
Then following \cite{Sej_Sripe_Gretton_Fuku:2013} and for every $x_0 \in \R^d$, we define
\begin{align}\label{eq.energy.based.kernel}
k_{\alpha}^{x_0} (x,y ) = \frac{1}{2}\croch{ \rho_{\alpha}(x,x_0) + \rho_{\alpha}(y,x_0) - \rho_{\alpha}(x,y) } ,
\end{align}
which is a positive semi-definite kernel leading to an RKHS $\mathcal{H}_{\alpha}^0$.
Plugging this in Eq.~\eqref{eq.energy.distance}, one can easily check that
\begin{align*}
  \mathcal{E}(X,Y;\alpha) & = 2 E\croch{ \rho_{\alpha}(X,x_0) + \rho_{\alpha}(Y,x_0) -2 k_{\alpha}^{x_0} (X,Y ) } \\
 & - E\croch{ \rho_{\alpha}(X,x_0) + \rho_{\alpha}(X^\prime,x_0) -2 k_{\alpha}^{x_0} (X,X^\prime ) } \\
& -  E\croch{ \rho_{\alpha}(Y,x_0) + \rho_{\alpha}(Y^\prime,x_0) -2 k_{\alpha}^{x_0} (Y,Y^\prime ) } \\
& = 2 E\croch{ k_{\alpha}^{x_0} (X,X^\prime ) + k_{\alpha}^{x_0} (Y,Y^\prime ) - 2 k_{\alpha}^{x_0} (X,Y )} \\
& = 2 \norm{ \mu_{P_X}^{\alpha} - \mu_{P_Y}^{\alpha} }_{\mathcal{H}_{\alpha}^0}^2 ,
\end{align*}
where $\mu_{P_X}^{\alpha}, \mu_{P_Y}^{\alpha}\in \mathcal{H}_{\alpha}^0$ respectively denote the mean elements of the distributions $P_X$ and $P_Y$, and $ \norm{\cdot}_{\mathcal{H}_{\alpha}^0} $ is the norm in $\mathcal{H}_{\alpha}^0$.

An important consequence of this derivation is that the exact and approximation algorithms described in Section~\ref{sec:algorithm} immediately apply to procedures relying on the optimization of the energy-based distance  $\mathcal{E}(X,Y;\alpha)$. 
This is all the more noticeable as our \emph{exact algorithm} has a lower memory complexity than the \emph{approximate optimization algorithm} of \ECP (with a similar time complexity). 
Therefore, for the same computational time, our exact optimization algorithm could replace the approximate algorithm used in \ECP. 
One can also emphasize that our approximate algorithm which is even faster could also be used (see its description in Section~\ref{sec:heuristic}).
This particular energy-based kernel, which is a characteristic kernel, has been involved in our simulation experiments as well (Section~\ref{sec:kernel-used}).

%

\section{New algorithms}\label{sec:algorithm}

In this section we first show how to avoid the preliminary calculation of the cost matrix required by \cite{harchaoui2007retrospective} to apply dynamic programming. 
The key idea is to compute the elements of the cost matrix on the fly when they are required by the dynamic programming algorithm. 
Roughly, this can be efficiently done by reordering the loops involved in Algorithm~\ref{algo:v1} proposed in \citep{harchaoui2007retrospective}. 
This leads to the new exact Algorithm~\ref{algo:v3}. It has a reduced space complexity of order $\mathcal{O}(n)$ compared to $\mathcal{O}(n^2)$ for the one used in \cite{harchaoui2007retrospective}. 
Note that including the constraint (introduced in Section~\ref{sec.constrained.segmentations}) on the segment sizes  mostly change the index of the \for loops in Algorithm~\ref{algo:v3}. We choose to describe the algorithm in the unconstrained version \eqref{def.penalized.criterion.global.first.step} to ease the understanding.

Second, we provide a faster but approximation algorithm (Section~\ref{sec:heuristic}), which enjoys a smaller complexity of order $\mathcal{O}(D_{\max}n)$ in time.  
It combines a low-rank approximation to the Gram matrix and the use of the binary segmentation heuristic (Section~\ref{sec.binary.segmentation}).
This approximation algorithm allows the analysis of very large signals ($n  \geq 10^6$).

\subsection{New efficient algorithm to recover the best segmentation from the Gram matrix}
\label{sec:exact-algorithm}

As exposed in Section~\ref{sec:model-select}, the main computational cost of the change-point detection procedure results from  Eq.~\eqref{def.penalized.criterion.global.first.step}, that is recovering the best segmentation with $1\leq D\leq \Dmax$ segments and solving
\begin{align} 
   \mathbf{L}_{D,n+1}  & =  \min_{\tau\in\T_{D}} \normHn{ Y-\muh}^2 & \mbox{(best fit to the data)} \nonumber \\ 
   \hat{m}_{D}  & =  \argmin_{\tau\in\T_{D}} \normHn{ Y-\muh }^2 & \mbox{(best segmentation)} \label{eq:opti_prob}
 \end{align} 
for every $1\leq D \leq D_{\max}$, where $\T_{D}$ denotes the collection of segmentations  of $\acc{1, \ldots, n }$ with $D$ segments.
This challenging step involves the use of dynamic programming  \citep{bellman61approximation,auger1989algorithms}, which provides the exact solution to the optimization problem \eqref{eq:opti_prob}.
Let us first provide some details on the usual way dynamic programming is implemented.

\subsubsection{Limitations of the standard dynamic programming algorithm for kernels}\label{sec.dynamic.programming.classical}

Let $\tau$ denote a segmentation in $D$ segments (with the convention that $\tau_1=1$ and $\tau_{D+1}= n+1$).
For any $1\leq d \leq D$, the segment $\{\tau_d, \ldots, \tau_{d+1}-1\}$ of the segmentation $\tau$ has a cost that is equal to
\begin{align} \label{eq.cost.expression.developed}
C_{\tau_d, \tau_{d+1}} = \sum_{ i=\tau_{d}}^{\tau_{d+1}-1} k(X_i, X_i) \ - \ \frac{1}{\tau_{d+1} - \tau_{d}}  \sum_{ i=\tau_{d}}^{\tau_{d+1}-1} \ \sum_{ j=\tau_{d}}^{\tau_{d+1}-1} k(X_i, X_j) .  
\end{align}
Then the cost of the segmentation $\tau$ is given by 
\begin{align*}
\Vert Y-\hat{\mu}^\tau\Vert^2_{\mathcal{H},n} = \sum_{d=1}^D C_{\tau_d, \tau_{d+1}} ,
\end{align*}
which is clearly \emph{segment additive} \citep{harchaoui2007retrospective,Arlot2011}.

Dynamic programming solves \eqref{eq:opti_prob} for all $1\leq D \leq D_{\max}$ by applying the following update rules
\begin{equation}\label{eq:updaterule} 
\forall \ 2\leq D \leq D_{\max},\quad \mathbf{L}_{D,n+1} = \min_{\tau \leq n} \{ \ \mathbf{L}_{D-1,\tau} + C_{\tau, n+1} \ \}, 
\end{equation}
which exploits the property that the optimal segmentation in $D$ segments over $\acc{1,\ldots,n}$ can be computed from optimal ones with $D-1$ segments over $\acc{1,\ldots,\tau}$ ($\tau \leq n$).
Making the key assumption that \emph{the cost matrix $\acc{C_{i,j}}_{1\leq i,j\leq n+1}$ has been stored}, we can compute $\mathbf{L}_{D,n+1}$ with Algorithm~\ref{algo:v1}. 
\begin{algorithm}[H]
\begin{algorithmic}[1]
\FOR{$D=2$ to $D_{\max}$}
\FOR{$\tau'=D$ to $n$}
\STATE $ \mathbf{L}_{D,\tau'+1} = \min_{\tau \leq \tau'} \{ \ \mathbf{L}_{D-1,\tau} + C_{\tau, \tau'+1} \ \}$
\ENDFOR
\ENDFOR
\caption{Basic use of Dynamic Programming}\label{algo:v1}
\end{algorithmic}
\end{algorithm}
This algorithm is used by \cite{harchaoui2007retrospective} and suffers two main limitations. First it assumes that the $C_{\tau,\tau'}$ have been already computed, and does not take into account the computational cost of its calculation. 
Second, it stores all $C_{\tau,\tau'}$ in a  $\O(n^2)$ matrix, 
which is memory expensive. 

A quick inspection of the algorithm reveals that the main step at Line 3 requires $\O(\tau^\prime)$ operations (assuming the $C_{i,j}$s have been already computed). 
Therefore, with the two \for loops we get a complexity of $\O(D_{\max}n^2)$ in time.
Note that without any particular assumption on the kernel $k(\cdot,\cdot)$,  computing $\normHn{ Y-\hat{\mu}^\tau}^2$ for a given segmentation $\tau$ is already of order $\O(n^2)$ in time since it involves summing over a quadratic number of terms of the Gram matrix (see Eq.~\eqref{eq.cost.expression.developed}). Therefore, there is no hope to solve \eqref{eq:opti_prob} exactly in less than quadratic time without additional assumptions on the kernel.

From Eq.~\eqref{eq.cost.expression.developed} let us also remark that computing each $C_{i,j}$ ($1\leq i< j \leq n$) naively requires itself a quadratic number of operations. Computing the whole cost matrix would require a complexity $\O(n^4)$ in time. 
Taking this into account, the dynamic programming step (Line~3 of Algorithm~\ref{algo:v1}) is not the limiting factor and the overall time complexity of Agorithm~\ref{algo:v1} is $\O(n^4)$. 

Finally, let us emphasize that this high computational burden is not specific of detecting change-points with kernels. It is rather representative of most learning procedures based on reproducing kernels and the associated Gram matrix \citep{Bach_2013}.

\subsubsection{Improved use of dynamic programming for kernel methods}\label{sec.improved.dynamic.programming}

\paragraph{Reducing space complexity}

From Algorithm~\ref{algo:v1}, let us first remark that each $C_{\tau,\tau^\prime}$ is used several times along the algorithm. 
%
%
A simple idea to avoid that is to swap the two \for loops in Algorithm~\ref{algo:v1}. 
This leads to the following modified Algorithm~\ref{algo:v2}, where each column $C_{\cdot,\tau^\prime+1}$ of the cost matrix is only used once unlike in Algorithm~\ref{algo:v1}. 
\begin{algorithm}[H]
\begin{algorithmic}[1]
\FOR{$\tau'=2$ to $n$}
\FOR{$D=2$ to $\min(\tau', D_{\max})$}
\STATE $ \mathbf{L}_{D,\tau'+1} = \min_{\tau \leq \tau'} \{ \ \mathbf{L}_{D-1,\tau} + C_{\tau, \tau'+1} \ \}$
\ENDFOR
\ENDFOR
\caption{Improved space complexity}\label{algo:v2}
\end{algorithmic}
\end{algorithm}

Importantly swapping the two \for loop does not change the output of the algorithm and does not induce any additional calculations.
Furthermore, at step $\tau^\prime$ of the first \for loop we do not need the whole $n \times n$ cost matrix to be stored, but only the column $C_{\cdot,\tau^\prime+1}$ of the cost matrix. This column is of size at most $\O(n)$.

Algorithm~\ref{algo:v2} finally requires storing coefficients $\acc{\mathbf{L}_{d,\tau} }_{1\leq d\leq D,\ 2\leq \tau \leq n}$ that are computed along the algorithm as well as successive column vectors $\acc{C_{\cdot,\tau} }_{2\leq \tau\leq n}$ (of size at most $n$) of the cost matrix.
This leads to an overall complexity of $\O(D_{\max}n)$ in space.
The only remaining problem is to compute these successive column vectors efficiently.
Let us recall that a naive implementation is prohibitive: each coefficient of the column vector can be computed in $\O(n^2)$, which would lead to $\O(n^3)$ to get the entire column.

\paragraph{Iterative computation of the columns of the cost matrix}

The last ingredient of our final exact algorithm is the efficient computation of each column vector $\acc{C_{\cdot,\tau} }_{2\leq \tau\leq n}$. Let us explain how to iteratively compute each vector in linear time.

First it can be easily observed that Eq.~\eqref{eq.cost.expression.developed} can be rephrased as follows
\begin{align*}
  C_{\tau,\tau^\prime} = \sum_{i=\tau }^{\tau^\prime-1} \paren{ k\paren{X_i, X_i} - \frac{A_{i,\tau^\prime}}{ \tau^\prime - \tau } } = D_{\tau,\tau^\prime} - \frac{1}{ \tau^\prime - \tau }  \sum_{i=\tau }^{\tau^\prime-1} A_{i,\tau^\prime},
\end{align*}
where $D_{\tau,\tau^\prime} =  \sum_{i=\tau }^{\tau^\prime-1} k\paren{X_i, X_i}$ and
\begin{align*}
  A_{i, \tau'} = - k(X_i, X_i) + 2 \sum_{j=i}^{\tau' -1} k(X_{i}, X_{j}) .
\end{align*}
Second, both $D_{\tau,\tau^\prime} $ and $\acc{A_{i,\tau^\prime}}_{i\leq \tau^\prime}$ can be iteratively computed from $\tau^\prime$ to $\tau^\prime+1$ by use of the two following equations:
\begin{align*}
  D_{\tau,\tau^\prime+1}  = D_{\tau,\tau^\prime} + k(X_{\tau^\prime},X_{\tau^\prime}) , \qquad \mbox{and}\qquad
  A_{i, \tau'+1} =   A_{i, \tau'} + 2 k(X_{\tau^\prime},X_{\tau^\prime}),\ \forall i\leq \tau' ,
\end{align*}
with $A_{\tau'+1,\tau'+1} = -k(X_{\tau^\prime+1},X_{\tau^\prime+1})$.
Therefore, as long as computing $k(x_i, x_j)$ requires $\O(1)$ operations, updating from $\tau'$ to $\tau'+1$ requires $\O\paren{\tau^\prime}$ operations.
\begin{rk}
		Note that for many classical kernels, computing $k(x_i, x_j)$ is indeed $\O(1)$ in time. 
		If $x_i \in \R^q$ with $q$ a positive integer being negligible with respect to other influential quantities such as $D_{\max}$ and $n$, several kernels such as the Gaussian, Laplace, or $\chi^2$ ones lead to a $\O(q) = \O(1)$ time complexity for evaluating $k(x_i, x_j)$. 
		By contrast in case where $q$ is no longer negligible, the resulting time complexity is multiplied by a factor $q$, which corroborates the intuition that the computational complexity increases with the ``complexity'' of the objects in $\X$.		
\end{rk}

This update rule leads us to the following Algorithm~\ref{algo:v3}, where each column $C_{\cdot,\tau'+1}$ in the first \for loop is computed only once:
\begin{algorithm}[H]
\begin{algorithmic}[1]
\FOR{$\tau^\prime=2$ to $n$}
\STATE Compute the ($\tau^\prime+1$)-th column $C_{\cdot,\tau^\prime+1}$ from  $C_{\cdot,\tau^\prime}$
\vspace{0.3cm}
\FOR{$D=2$ to $\min(\tau^\prime, D_{\max})$}
\STATE $ \mathbf{L}_{D,\tau^\prime+1} = \min_{ \tau \leq \tau^\prime} \{ \mathbf{L}_{D-1,\tau} + C_{\tau, \tau^\prime+1}  \}$
\ENDFOR
\vspace{0.3cm}
\ENDFOR
\caption{Improved space and time complexity (\Kern)}\label{algo:v3}
\end{algorithmic}
\end{algorithm}

From a computational point of view, each step of the first \for loop in Algorithm~\ref{algo:v3} requires $\O(\tau')$ operations to compute $C_{\cdot,\tau^\prime+1}$ and at most $\O(D_{\max}\tau')$ additional operations to perform the dynamic programming step at Line~4.
Then the overall complexity is $\O(D_{\max}n^2)$ in time and $\O(D_{\max}n)$ in space. 
This should be compared to the $\O(D_{\max}n^4)$ time complexity of the naive calculation
of the cost matrix and to the $\O(n^2)$ space complexity of the standard Algorithm~\ref{algo:v1} from \cite{harchaoui2007retrospective}.


\subsubsection{Runtimes comparison to other implementations}

\label{sec.runtimes.comparison}

The purpose of the present section is to perform the comparison between Algorithm~\ref{algo:v3} and other competitors to illustrate their performances as the sample size increases with $D_{\mathrm{max}}=100$.

The first comparison has been carried out between Algorithm~\ref{algo:v3} and  the naive quartic computation of the cost matrix (Algorithm~\ref{algo:v1}). These two algorithms have been implemented in C and packaged in R.
Results for these algorithms are reported in Figure~\ref{fig:time1} (Left). 
Unsurprisingly, our quadratic algorithm (called \Kern) is faster than a quartic computation of the cost matrix (called KCP) even for very small sample sizes ($n < 320$).

Second, we also compared the runtime of \Kern (Algorithm~\ref{algo:v3}) with that of \ECP discussed in Section~\ref{sec.link.other.approaches} implemented in the R-package \cite{ecp} (see the middle panel of Figure~\ref{fig:time1}).
Since \ECP is based on the binary segmentation heuristic applied to an energy-based distance, its worst-case complexity is at most $\O(D_{\mathrm{max}}n^2)$ in time, which is the same as that of \Kern.
Note also that the native implementation of \ECP involves an additional procedure relying on permutations to choose the number of change-points. 
If $B$ denotes the number of permutations, the induced complexity is then $\O(B D_{\mathrm{max}}n^2)$ in time. 
To be fair, we compared our approach and \ECP with and without the permutation layer. 
Finally it is also necessary to emphasize that unlike \Kern, \ECP does not provide the exact but only an approximate solution to the optimization problem \eqref{eq:opti_prob}.
Results are summarized in Figure~\ref{fig:time1} (Middle). 
It illustrates that our exact algorithm (\Kern) has a quadratic complexity similar to that of \ECP with and without permutations. Our algorithm is the overall fastest one even for small sample size ($n<1\,000$). Although this probably results from implementation differences, it is still noteworthy since \Kern is exact unlike \ECP.

Finally, Figure~\ref{fig:time1}~(Right) illustrates the worse memory use of \ECP (with and without any permutations) as compared to that of the exact \KernSegGauss (\Kern used with the Gaussian kernel). \ECP has an $\O(n^2)$ space complexity, while \KernSegGauss is $\O(n)$. For $n$ larger than $10^4$ the quadratic space complexity of \ECP is a clear limitation since several Gb of RAM are required.

\begin{figure}
\begin{center}
	\hspace*{-1.2cm}
\begin{tabular}{ccc}
KernSeg and KCP & KernSeg and ECP (Time) & KernSeg and ECP (Memory)\\
\includegraphics[scale=0.32, clip=true, trim= 0cm 0cm 1cm 2cm]{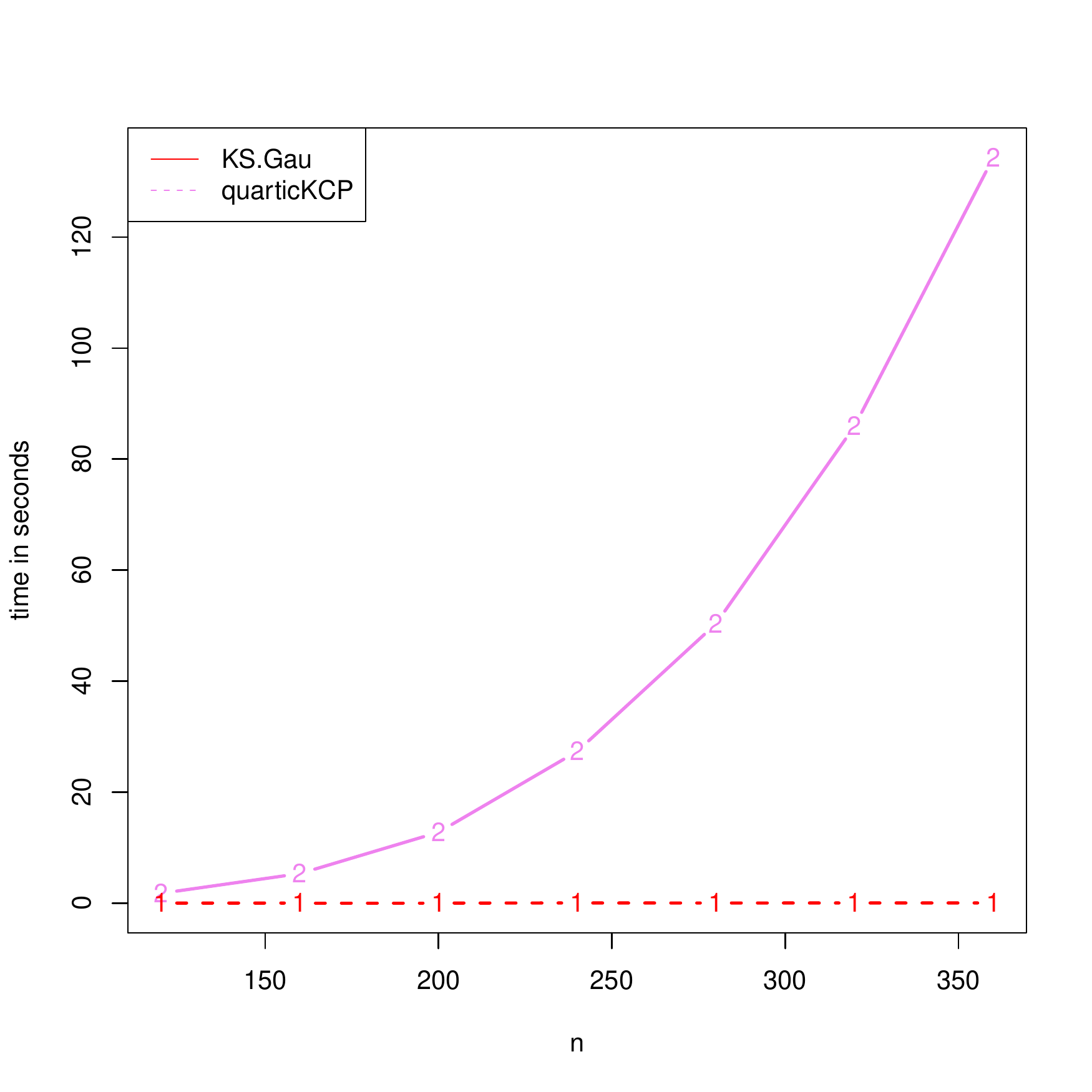} &
\includegraphics[scale=0.32, clip=true, trim= 0cm 0cm 1cm 2cm]{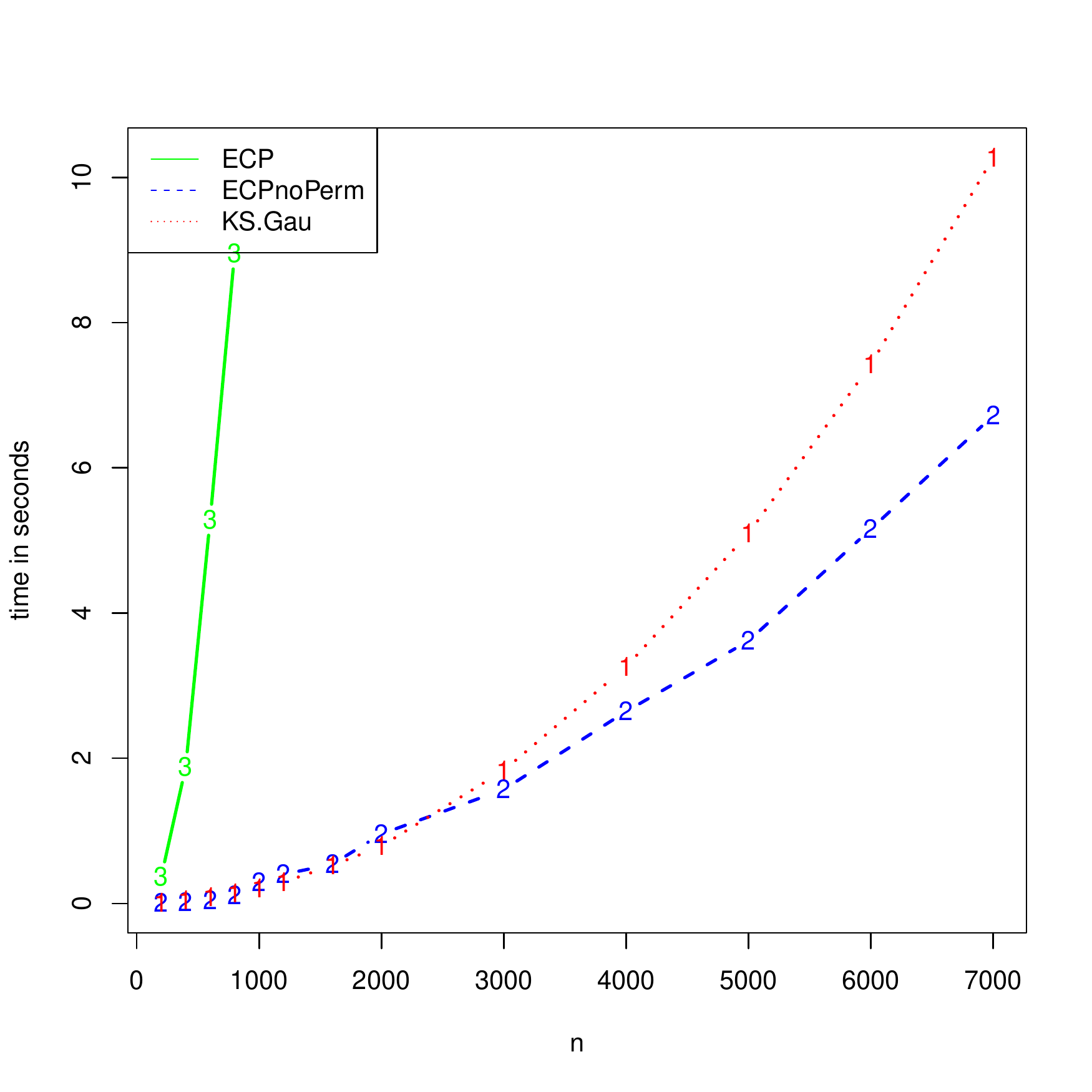} &
\includegraphics[scale=0.32, clip=true, trim= 0cm 0cm 1cm 2cm]{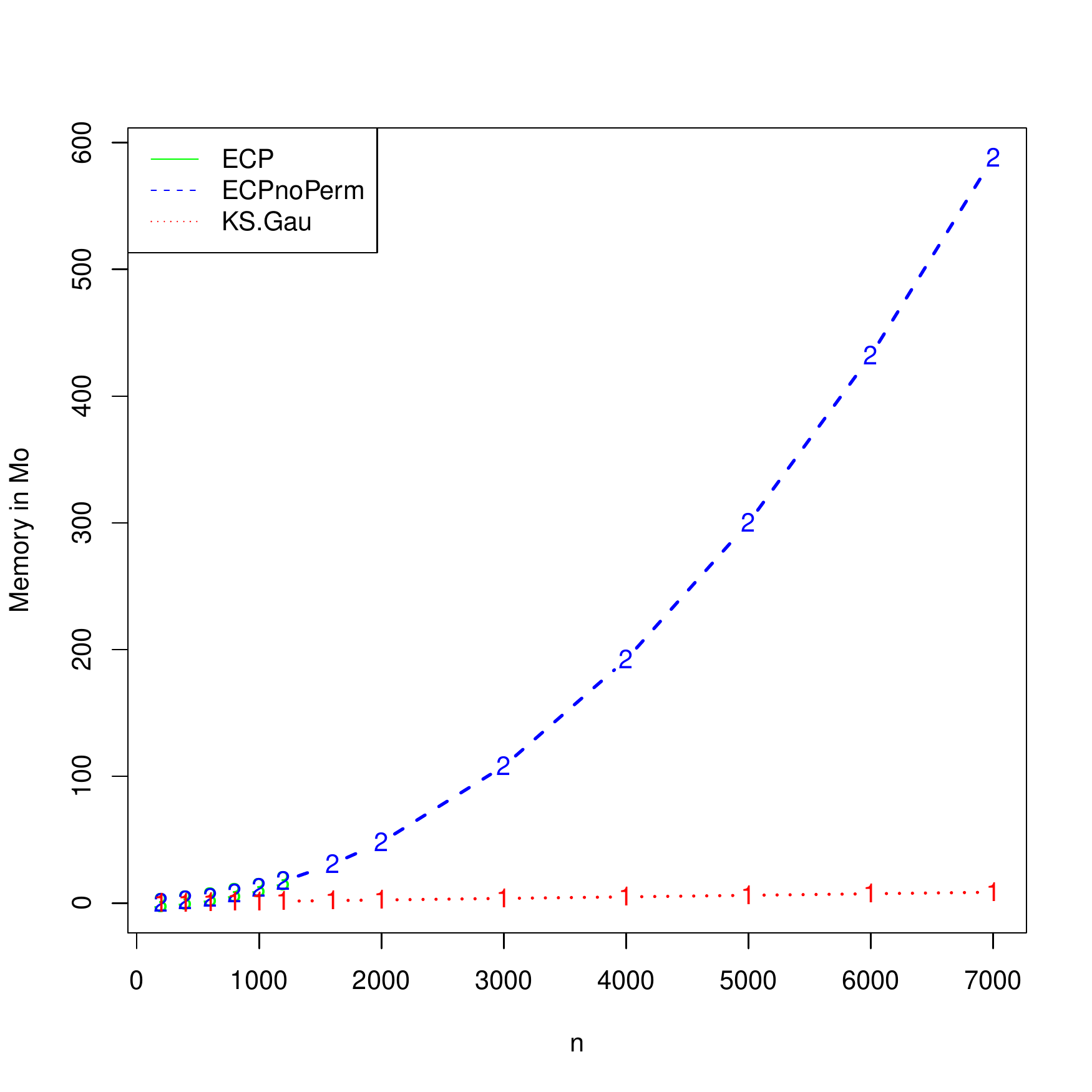} 
\end{tabular}
\end{center}
\caption{(Left) Runtime in seconds of Algorithm~\ref{algo:v3} as a function of the length of the signal ($n$) for $D_{\mathrm{max}}=100$. (1-red) and a quartic computation of the cost matrix (2-violet).
(Middle) Runtime in seconds of Algorithm~\ref{algo:v3} as a function of the length of the signal (1-red) and of \ECP without permutation (2-blue) and \ECP with the default number of permutations (3-green). (Right) Memory in mega-octet of Algorithm~\ref{algo:v3} as a function of the length of the signal (1-red) and of \ECP without permutation (2-blue) and \ECP with the default number of permutations (3-green). The performances of \ECP with or without permutation are exactly the same.}\label{fig:time1}
\end{figure}

\subsection{Approximating the Gram matrix to speed up the algorithm}\label{sec:heuristic}

In Section~\ref{sec.improved.dynamic.programming}, we described an improved algorithm called \Kern, which carefully combines dynamic programming with the computation of the cost matrix elements. This new algorithm (Algorithm~\ref{algo:v3}) provides the exact solution to the optimization problem given by Eq.~\eqref{eq:opti_prob}.
However without any further assumption on the underlying reproducing kernel, this algorithm only achieves the complexity $\O(n^2)$ in time, which is a clear limitation with large scale signals ($n \geq 10^5$). 
Note also that this limitation results from the use of general positive semi-definite kernels (and related Gram matrices) and cannot be improved by existing algorithms to the best of our knowledge. 
For instance, the binary segmentation heuristic \citep{Fryzlewicz2012}, which is known to be computationally efficient for parametric models, suffers the same $\O(n^2)$ time complexity when used in the reproducing kernel framework (see Section~\ref{sec.binary.segmentation}).

Let us remark however that for some particular kernels it is possible to reduce this time complexity. For example for the linear one $k(x,y) = \scal{x , y}_{\R^d},\ x,y\in\R^d$, one can use the following trick
\begin{align}\label{eq.computational.trick}
\sum_{1\leq i\neq j \leq n} k(X_i,X_j) & = \sum_{1\leq i\leq n}\scal{ X_i ,  \sum_{j=1}^n X_j - X_i }   =  \norm{ \sum_{i=1}^n X_i }^2 - \sum_{i=1}^n \norm{ X_i}^2 ,
\end{align}
where $\norm{\cdot}$ denotes the Euclidean norm in $\R^d$.

The purpose of the present section is to describe a versatile strategy (\textit{i.e} applicable to any kernel) relying on a low-rank approximation to the Gram matrix \citep{Williams01usingthe,Smola_2000,Fine01efficientsvm}. 
This approximation allows to considerably reduce the computation time by exploiting \eqref{eq.computational.trick}.
Note however that the resulting procedure achieves this lower time complexity at the price of only providing an approximation to the exact solution to \eqref{eq:opti_prob} (unlike the algorithm described in Section~\ref{sec.improved.dynamic.programming}).

\subsubsection{Low-rank approximation to the Gram matrix} \label{sec.low.rank.approx}

The main idea is to follow the same strategy as the one described by \cite{DriMah_2005} to derive a low-rank approximation to the Gram matrix $\K = \acc{\K_{i,j}}_{1\leq i,j\leq n}$, where $\K_{i,j} = k(X_i, X_j)$.

Assuming $\K$ has rank $\mathrm{rk}(\K) \ll n$, we could be tempted to compute the best rank approximation to $\K$ by computing the $\mathrm{rk}(\K)$ largest eigenvalues (and corresponding eigenvectors) of $\K$.
However such computations induce a $\O(n^3)$ time complexity which is prohibitive.

Instead, \cite{DriMah_2005} suggest applying this idea on a square sub-matrix of $\K$ with size $p \ll n$.
For any subsets $I,J \subset \acc{1, \ldots, n }$, let $\K_{I, J}$ denote the sub-Gram matrix with respectively row and column indices in $I$ and $J$.
Let $J_p \subset \acc{1, \ldots, n }$ denote such a subset with cardinality $p$, and consider the sub-Gram matrix $ \K_{J_p,J_p}$ which is of rank $r \leq p $.
Further assuming $r=p$, the best rank $p$ approximation to $\K_{J_p,J_p}$ is $\K_{J_p,J_p}$ itself.
This leads to the final approximation to the Gram Matrix $\K$ \citep{DriMah_2005,Bach_2013} by
\begin{align*}
 \widetilde{\K} = \K_{I_n, J_p} \ \K^+_{J_p, J_p} \ \K_{J_p, I_n}, 
\end{align*}
where $I_n = \acc{1,\ldots, n }$, and $\K^+_{J_p, J_p}$ denotes the pseudo-inverse of $\K_{J_p, J_p}$.
Further considering the SVD decomposition of $ \K_{J_p, J_p} = \mathbf{U}' \Lambda \mathbf{U}$, for an orthonormal matrix $\mathbf{U}$, we can rewrite 
\begin{align*} \widetilde{\K} =  \Z' \Z, \qquad \mbox{with}  \quad  \Z =   \Lambda^{-1/2} \mathbf{U} \ \K_{J_p, I_n}  \in \mathcal{M}_{p,n}(\R) . \hfill
\end{align*}
Note that the resulting time complexity is $\O(p^2n)$, which is smaller than the former $\O(n^3)$ as long as $p=o(\sqrt{n})$.
This way, columns $\acc{Z_i}_{1\leq i\leq n}$ of $\Z$ act as new $p$-dimensional observations, and each $\widetilde{\K}_{i,j}$ can be seen as the inner product between two vectors of $\R^p$, that is
\begin{align} \label{eq.finite.dimensional.vectors}
\widetilde{\K}_{i,j} =  Z_i' Z_j .
\end{align}
The main interest of this approximation is that, using Eq.~\eqref{eq.computational.trick}, computing the cost of a segment of length $t$ has a complexity $\O(t)$ in time unlike the usual $\O(t^2)$ that holds with general kernels.
%

%
Interestingly such an approximation to the Gram matrix can be also built from a set of deterministic points in $\X$. This remark has been exploited to compute our low-rank approximation for instance in the simulation experiments as explained in Section~\ref{sec:approx}.

Note that choosing the set $J_p$ of columns/rows leading to the approximation $\widetilde \K$ is of great interest in itself for at least two reasons. 
First from a computational point of view, the $p$ columns have to be selected following a process that does not require to compute the $n$ possible columns beforehand (which would induce an $O(n^2)$ time complexity otherwise). 
Second, the quality of $\widetilde \K$ to approximate $\K$ crucially depends on the rank of $\widetilde \K$ that has to be as close as possible to that of $\K$, which remains unknown for computational reasons.
However such questions are out of scope of the present paper, and we refer interested readers to \cite{Williams01usingthe,DriMah_2005,Bach_2013} where this point has been extensively discussed.

\subsubsection{Binary segmentation heuristic}\label{sec.binary.segmentation}

Since the low-rank approximation to the Gram matrix detailed in Section~\ref{sec.low.rank.approx} leads to finite dimensional vectors in $\R^p$ \eqref{eq.finite.dimensional.vectors}, the change-point detection problem described in Section~\ref{sec:statistical-model} amounts to recover abrupt changes of the mean of a $p$-dimensional time-series.
Therefore any existing algorithm usually used to solve this problem in the $p$-dimensional framework can be applied.
An exhaustive review of such algorithms is out of the scope of the present paper. 
However we will mention only a few of them to highlight their drawbacks and motivate our choice.
Let us also recall that our purpose is to provide an efficient algorithm allowing: $(i)$ to (approximately) solve Eq.~\eqref{eq:opti_prob} for each $1\leq D\leq D_{\max}$ and $(ii)$ to deal with large sample sizes ($n\geq 10^6$).

\medskip

The first algorithm is the usual version of constrained dynamic programming \citep{auger1989algorithms}. Although it has been recently revisited with $p=1$ by \cite{rigaill2015pruned,cleynen2014segmentor3isback,maidstone2017optimal}, it has a $\O(n^2)$ time complexity with $p>1$ , which excludes dealing with large sample sizes.
Another version of regularized dynamic programming has been explored by \cite{killick2012optimal} who designed the PELT procedure. It provides the best segmentation over all segmentations with a penalty of $\lambda$ per change-point with an $\O(n)$ complexity in time if the number of change-points is linear in $n$. 
Importantly, the complexity of the pruning inside PELT depends on the true number of change-points. For only a few change-points, the PELT complexity remains quadratic in time.
With PELT, it is not straightforward to efficiently solve Eq.~\eqref{eq:opti_prob} for each $1\leq D\leq D_{\max}$, which is precisely the goal we pursue. 
Note however that it would still be possible to recover some of those segmentations by exploring a range of $\lambda$ values like in CROPS \citep{haynes2017computationally}.

A second possible algorithm is the so-called \emph{binary segmentation} \citep{olshen04circular,yang12simple,Fryzlewicz2012} that is a standard heuristic for approximately solving Eq.~\eqref{eq:opti_prob} for each $1\leq D\leq D_{\max}$.
This iterative algorithm computes the new segmentation $\widetilde{\tau}\paren{D+1}$ with $D+1$ segments from $\widetilde{\tau}\paren{D}$ by splitting one segment of $\widetilde{\tau}\paren{D}$ into two new ones without modifying other segments.
More precisely considering the set of change-points $\widetilde{\tau}\paren{D} = \acc{\tau_1, \ldots, \tau_{D+1} }$, binary segmentation provides
\begin{align*} 
\widetilde{\tau}\paren{D+1}  & =  \argmin_{\tau \in \mathcal{T}_{D+1} | \tau \cap \widetilde{\tau}\paren{D} = \widetilde{\tau}\paren{D} } \acc{ \Vert Y-\hat{\mu}^\tau\Vert^2_{\mathcal{H},n} } .\hfill 
 \end{align*} 
Since only one segment of the previous segmentation is divided into two new segments at each step, the binary segmentation algorithm provides a simple (but only approximate) solution to Eq.~\eqref{eq:opti_prob} for each $1 \leq D \leq D_{\max}$.

We provide some pseudo-code for binary segmentation in Algorithm \ref{heur:v2}. 
It uses a sub-routine described by Algorithm \ref{heur:v1} to compute the best split of any segment $[\tau, \tau'[$ of the data. To be specific, this BestSplit routine outputs
four things: (1) the reduction in cost of spliting the segment $[\tau, \tau'[$, (2) the best change to split (3) the resulting left segment and (4) the resulting right segment. 

In the binary segmentation algorithm candidate splits are stored and handled using a binary heap data structure \cite{cormen2009introduction} using the reduction in cost as a key.
This data structure allows to efficiently
insert new splits and extract the best split in $\O(\log(D_{max}))$ at every time step.
Without such a structure inserting splits and extracting the best split would typically be in $\O(D_{max})$ and for large
$D_{max}$ the binary segmentation heuristic is at best $\O(n^2)$.
Note that the \RBS procedure \cite{Pierre-Jean08092014}, which is involved in our simulation experiments (Section~\ref{sec:rbs-description}), also uses this heuristic.

\begin{algorithm}[H]
	\begin{algorithmic}[1]
		\STATE $\hat{m} = \underset{ \tau < t < \tau'}{\min} \{ C_{\tau, t} + C_{t, \tau'} \}$ and $\hat{t} = \underset{ \tau < t < \tau'}{\arg \min} \{ C_{\tau, t} + C_{t, \tau'} \}$
		\STATE Output four things  (1) $C_{\tau, t} - \hat{m}$, (2) $\hat{t} \ $, (3) $[\tau, \hat{t}[$ and (4) $[\hat{t}, \tau'[$
		\caption{BestSplit of segment $[\tau, \tau'[$}\label{heur:v1}
	\end{algorithmic}	
\end{algorithm}
\begin{algorithm}[H]
	
	\begin{algorithmic}[1]
		\STATE Segs $= \{ [1, n+1[ \}$
		\STATE Changes $= \emptyset$ 
		\STATE CandidateSplit $= \emptyset$ [a binary heap]
		\FOR{ $D_{max}$ iteration}
		\FOR {$ aseg \in Segs $}
		\STATE Insert BestSplit($aseg$) in CandidateSplit
		\ENDFOR
		\STATE Extract the best split of CandidateSplit and recover: $\hat{t}$, $[\tau, \hat{t}[$ and $=[\hat{t}, \tau'[$
		\STATE Add $ \hat{t}$ in Changes
		\STATE Set Segs to  $\{ \ [\tau, \hat{t}[, [\hat{t}, \tau'[ \ \}$
		\ENDFOR
		\caption{Binary Segmentation}\label{heur:v2}
	\end{algorithmic}
\end{algorithm}

Assuming the best split of any segment is linear in its length the overall time complexity of binary segmentation for recovering approximate solutions to \eqref{eq:opti_prob} for all $1 \leq D \leq D_{\max}$ is around $\O\paren{ \log(D_{\max}) n}$ in practice. 
The worst case time complexity is $\O\paren{ D_{\max} n}$. A typical setting where it is achieved is with the linear kernel when $i \mapsto X_i=\exp(i)$ for instance. At the $i$-th iteration of the binary segmentation algorithm, the best split of a segment of length $n-i+1$ corresponds to one segment of length $1$ and another one of length $n-i$.

An important remark is that  binary segmentation only achieves this reduced $\O\paren{ \log(D_{\max}) n}$ time complexity provided that recovering the best split of any segment is linear in its length. This is precisely what has been allowed by the low-rank matrix approximation summarized by Eq.~\eqref{eq.finite.dimensional.vectors}. 
Indeed with the low-rank approximation, computing the best split of any segment is linear in $n$ and $p$. 
The resulting time complexity of binary segmentation is thus $\O\paren{ p\log(D_{\max}) n}$, which reduces to $\O\paren{ \log(D_{\max}) n}$ as long as $p$ is small compared to $n$.
By contrast without the approximation recovering the best split is typically quadratic in the length of the segment and binary segmentation would suffer an overall time complexity of order $\O(\log(D_{\max})n^2)$ or $\O(D_{\max}n^2)$.

\subsubsection{Implementation and runtimes of the approximate solution}
\label{sec.runtimes.approx.solution}

The approximation algorithm we recommend is the combination of the low-rank approximation step detailed in Section~\ref{sec.low.rank.approx} and of the binary segmentation discussed in Section~\ref{sec.binary.segmentation}.
The resulting time complexity is then $\O( p^2n + p \log(D_{\max}) n) $, which allows dealing with large sample sizes $(n\geq 10^6)$.

From this time complexity it arises that
an influential parameter is the number $p$ of columns of the matrix used to build the low-rank approximation.
In particular this low-rank approximation remains computationally attractive as long as $p=o(\sqrt{n})$.
Figure~\ref{fig:time2} illustrates the actual time complexity of this fast algorithm (implemented in C) with respect to $n$ for various values of $p$: $(i)$ a constant value of $p$ and $(ii)$  $p=\sqrt{n}$. 
To ease the comparison, we also plotted the runtime of the exact algorithm (Algorithm~\ref{algo:v3}) detailed in Section~\ref{sec.improved.dynamic.programming} and \RBS that uses binary segmentation (see Section~\ref{sec:rbs-description}).

Our fast approximation algorithm (\aKern) recovers a quadratic complexity if $p=\sqrt{n}$. However its overhead is much smaller than that of the exact algorithm, which makes it more applicable than the latter with large signals in practice. 
Note also that Figure~\ref{fig:time2} illustrates that \aKern returns the solution  in a matter of seconds with a sample size of $n= 10^5$, which is much faster than \Kern (based on dynamic programming) that requires a few minutes.
The \RBS implementation involves preliminary calculations which make it slower than \aKern with $n \leq 2\cdot 10^3$. However for larger values ($n\geq 10^4$) \RBS is as fast as \aKern with $p=10$.

\begin{figure}
\begin{center}
\begin{tabular}{c}
KernSeg Exact and Heuristic\\
\includegraphics[scale=0.4]{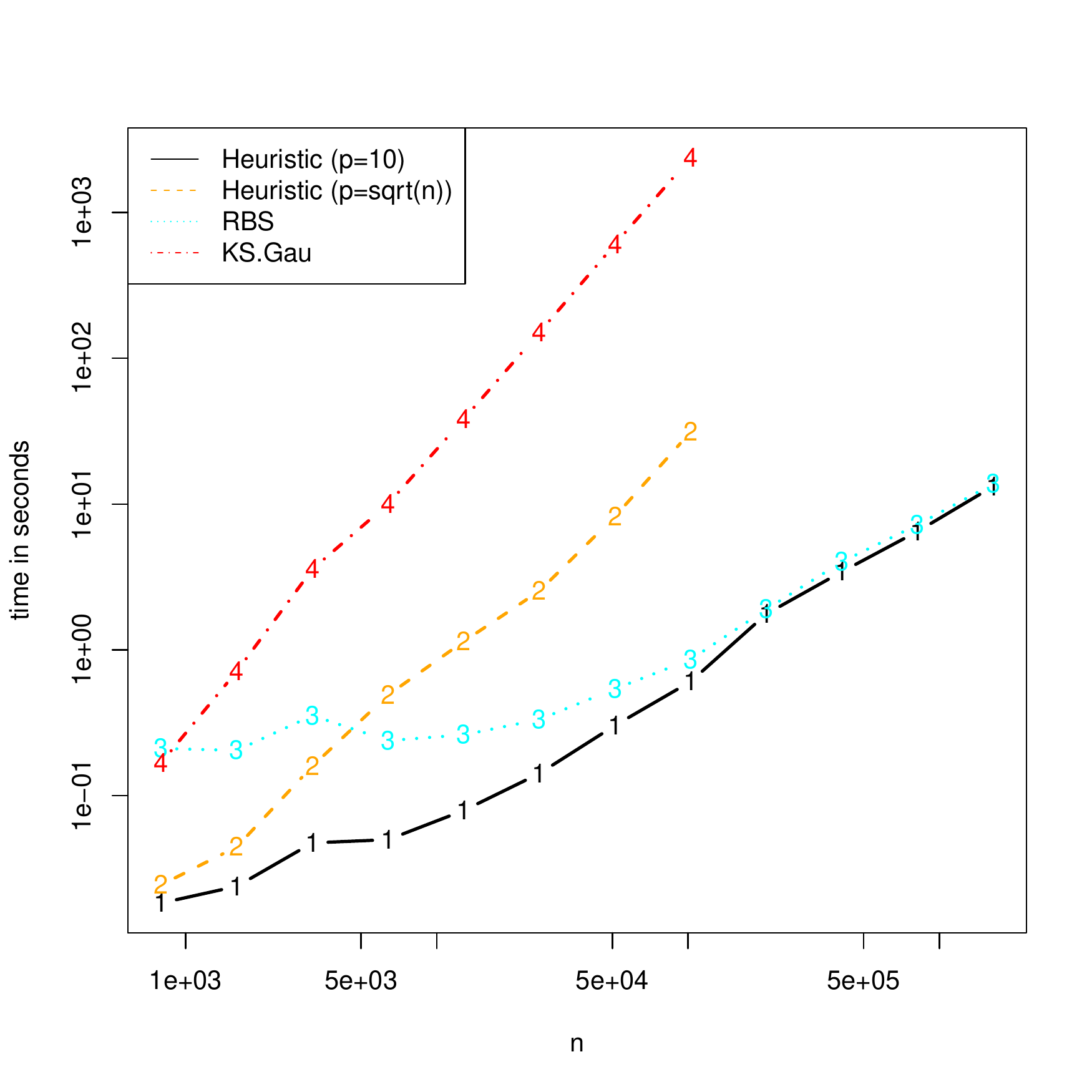}
\end{tabular}
\end{center}
\caption{Runtime as a function of $n$ (length of the signal) for $D_{max}=100$.
	Runtime of our approximation algorithm with $p=100$ (1-black) and $p=\sqrt{n}$ (2-orange), \RBS (3-cyan), exact Algorithm~\ref{algo:v3} (4-red)}\label{fig:time2}
\end{figure}


\section{Segmentation assessment}
\label{sec:eval-segm}

From a statistical point of view \Kern provides the same performance as that of \cite{Arlot2011}. However it greatly improves on the latter in terms of computational complexity as proved in Section~\ref{sec:exact-algorithm}.
In their simulation experiments \cite{Arlot2011} mainly focus on detecting change-points in the distribution of $\R$-valued data as well as of more structured objects such as histograms.
Here we rather investigate the performance of the kernel change-point procedure on specific two-dimensional biological data: the DNA copy number and the BAF profiles (see Section~\ref{sec.DNA.compy.number.data}).
More precisely our experiments highlight two main assets of applying reproducing kernels to these biological data: $(i)$ reasonable kernels avoid the need for modeling the type of change-points we are interested in and improve upon state-of-the-art approaches in this biological context, and $(ii)$ the high flexibility of kernels facilitates data fusion, that is allows to combine different data-types and get more power to detect true change-points.

\medskip

In the following we first briefly introduce the type of data we are looking at, and describe our simulation experiments obtained by resampling from a set of real annotated DNA profiles.
Second, we provide details about the change-point procedures involved in our comparison.
We also define the criteria used to assess the performance of the estimated segmentations.
Finally, we report and discuss the results of these experiments.

\subsection{Data description}

\subsubsection{DNA copy number data} \label{sec.DNA.compy.number.data}
DNA copy number alterations are a hallmark of cancer cells \cite{hanahan11hallmarks}. The accurate detection and interpretation of such changes are two important steps toward improved diagnosis and treatment.  
%
Normal cells have two copies of DNA, inherited from each biological parent of the individual. In tumor cells,  parts of a chromosome of various sizes (from kilobases to a chromosome arm) can be deleted, or copied several times.  As a result, DNA copy numbers in tumor cells are piecewise constant along the genome. 
Copy numbers can be measured using microarray or sequencing experiments. Figure~\ref{fig:copy-number-data-c-b} displays an example of copy number profiles that can be obtained from SNP-array data \citep{neuvial11statistical}. 

The left panel (denoted by TCN) represents estimates of the total copy number (TCN). The right panel (denoted by BAF) represents estimates of allele B fractions (BAF) using only homozygous position.
We refer to \cite{neuvial11statistical} for an explanation of how these estimates are obtained. 
In the normal region [0-2200], TCN is centered around two copies and BAF has three modes at 0, 1/2 and 1.

On top of Figure~\ref{fig:copy-number-data-c-b}, numbers $(a,b)$ represent each parental copy number in the corresponding segment. For instance $(0,2)$ means that the total number of copies in the segment is 2. But a copy from one of the two parents is missing while the other copy has been duplicated.
%
%
\begin{figure}[!h]
  \centering
  \includegraphics[width=8cm]{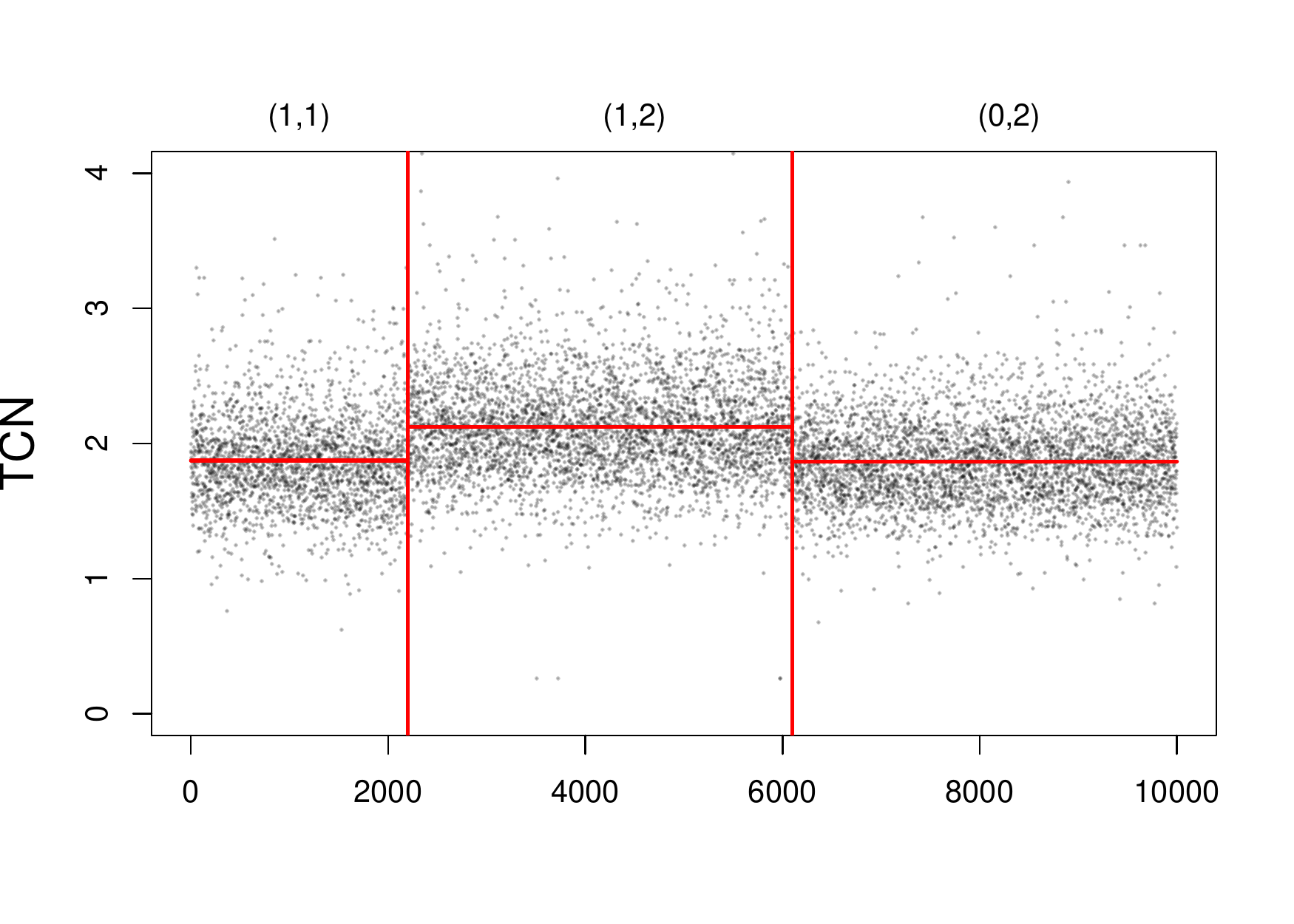}
  \includegraphics[width=8cm]{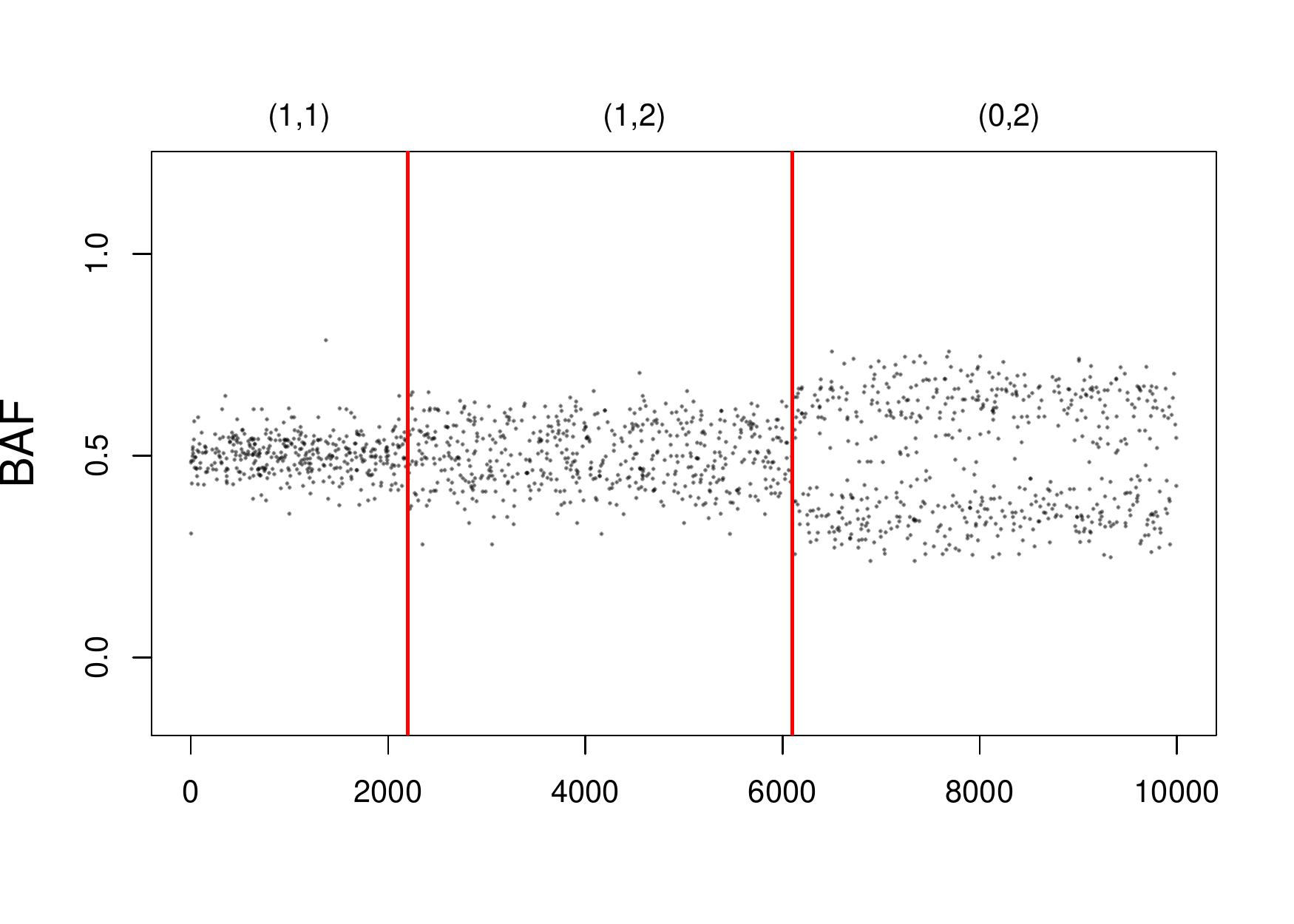}
\caption{SNP array data. Total copy numbers (TCN), allelic ratios (BAF) along 10,000 genomic loci. Red vertical lines represent change-points, and red horizontal lines represent estimated mean signal levels between two change-points. 
}
\label{fig:copy-number-data-c-b}
\end{figure}
Importantly any change in only one of the parental copy numbers is reflected in both TCN and BAF.
Therefore it makes sense to jointly analyze both dimensions to ease the identification of change-points.

Importantly in the following, allelic ratios (BAF) are always symmetrized (or folded)-- that is we consider $|BAF - 0.5|$--to facilitate the segmentation task. This is common practice in the field~\cite{staaf08segmentation}.

\subsubsection{Generated data}
\label{sec:simulation-data}
Realistic DNA profiles with known truth (similar to that of Figure~\ref{fig:copy-number-data-c-b}) have been generated using the \code{acnr} package \citep{Pierre-Jean08092014}.
The constituted benchmark consists of profiles with $5,000$ positions of heterozygous SNPs and exactly $K=10$ change-points.
As in \cite{Pierre-Jean08092014} we only consider four biological states for the segments.
%
The \code{acnr} package allows to vary the difficulty level by adding normal cell contamination, thus degrading tumor percentage.  
Three levels of difficulty have been considered by varying tumor percentage: 100\% (easy case), 70\%, and 50\% (difficult case). 
Figure~\ref{fig:example-data} displays three examples of simulated profiles (one for each tumor purity level). 

%
\begin{figure}[ht]
  \centering
  \includegraphics[width=0.79\columnwidth]{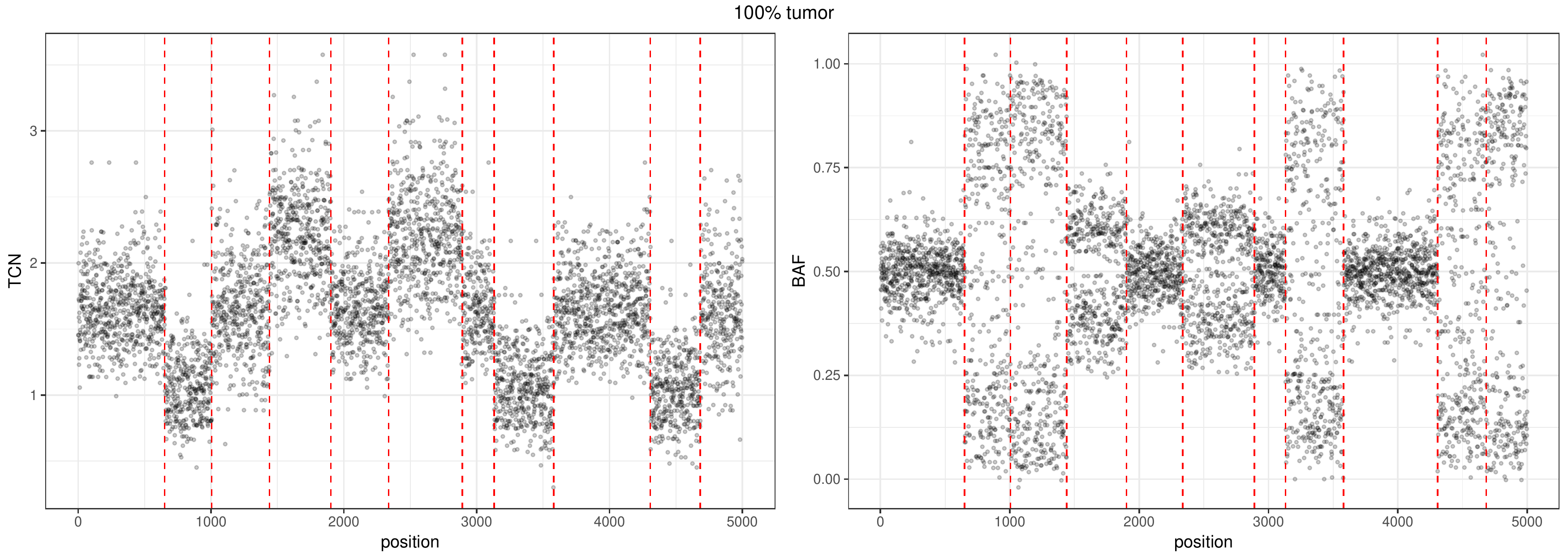}
  \includegraphics[width=0.79\columnwidth]{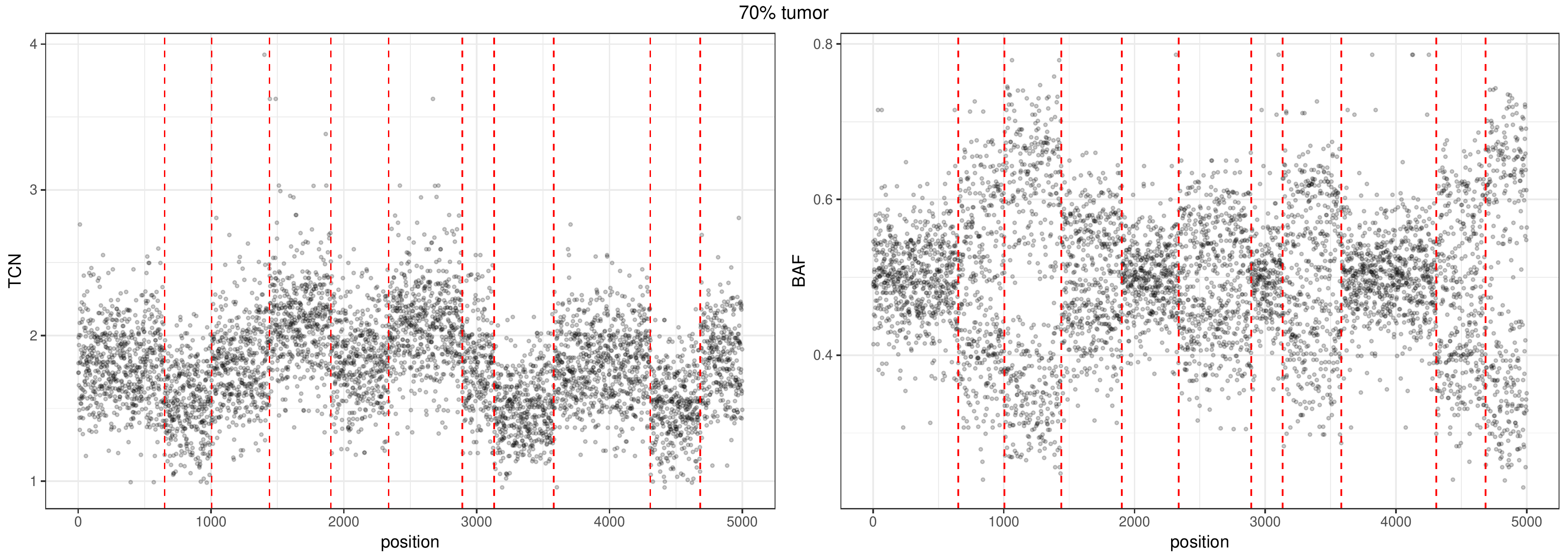}
  \includegraphics[width=0.79\columnwidth]{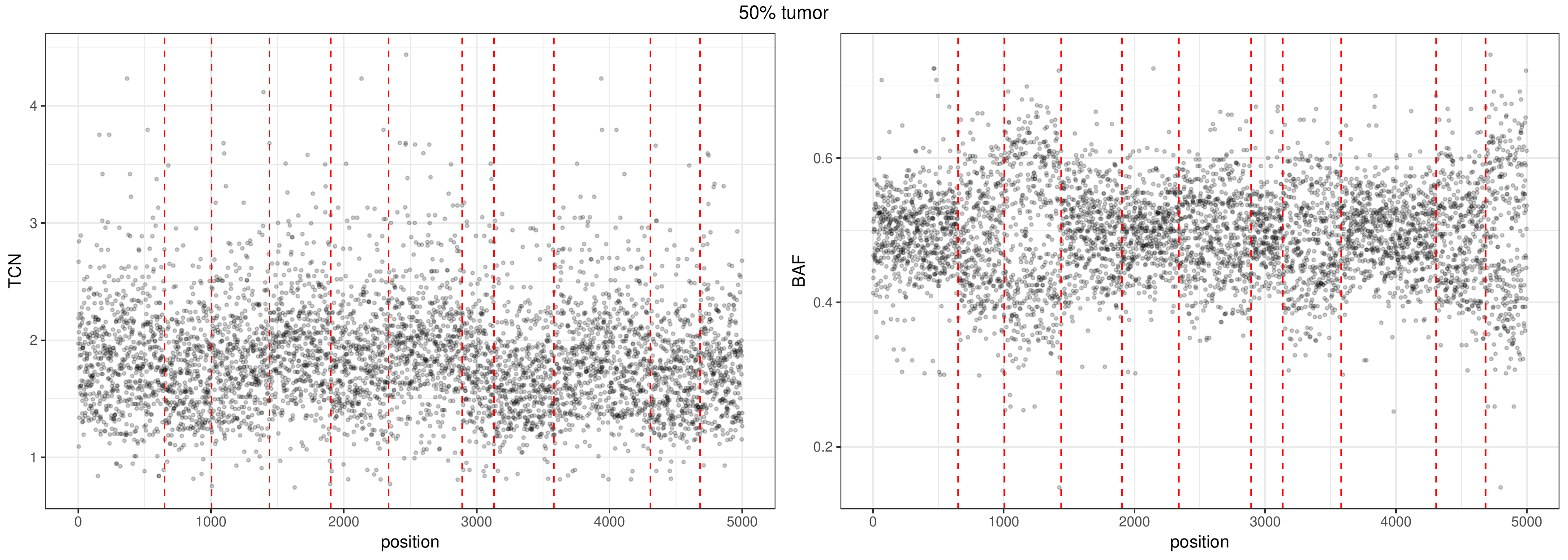}
  \caption{Benchmark 1:  Profiles simulated with the \code{acnr} package. Each line corresponds to a tumor percentage (100\%, 70\% and 50\%). The first column correspond to copy number (TCN) and the second to the allele B fraction (BAF).}
  \label{fig:example-data}
\end{figure}
For each level, $N=50$ profiles (with the same segment states and change-points) are generated making a total of $150$ simulated profiles both for BAF and TCN.
%





\subsection{Competing procedures}
\label{sec:competing-proc}

\subsubsection{\Kern and \aKern}
\label{sec:kernel-used}
The implemented exact algorithm \Kern (corresponding to Algorithm~\ref{algo:v3}) and its fast approximation \aKern (based on binary segmentation) are applicable with any kernel (Gaussian, exponential, polynomial, \ldots).
In our simulation experiments we consider three kernels. 
\begin{itemize}
	\item The first one is the so-called linear kernel defined by $k(x,y) = \scal{x,y}$, where $x,y \in\R$. It is used as a benchmark since \Kern with this kernel reduces to the procedure of \cite{Lebarbier2005}. The corresponding procedure is denoted by \KernSegLin.
	
	\item The second one is the Gaussian kernel defined for every $x,y\in\R$ by
	\begin{align*}
	k_{\delta}(x,y)= \exp\croch{ \frac{ - \abs{ x-y}^2}{\delta}}, \qquad \forall \delta>0 .
	\end{align*}
	Since it belongs to the class of characteristic kernels \citep{SriFuLa_2010}, it is a natural choice to detect any abrupt changes arising in the full distribution \cite{Arlot2011}.  
We call this procedure \KernSegGauss.
	
	\item The third one is the kernel associated with the energy-based distance introduced in Eq.~\eqref{eq.energy.based.kernel} with $\alpha=1$ and $x_0=0$. This particular choice is the prescribed value in the ECP package \citep{ecp}. 
	We call this procedure \KernSegECP.
\end{itemize}
These three kernels allow to  $(i)$ illustrate the interest of characteristic kernels compared to non characteristic ones, and $(ii)$ assess the performances of change-point detection with kernels (\Kern) compared to other approaches (\ECP, \RBS).
However other characteristic kernels such as the Laplace or exponential ones (see Section~\ref{sec:characteristic}) could have been considered as well.

For all kernels we considered $D_{max}=100$.
Note also that for all approaches and for both TCN and BAF profiles we first scaled the data using a difference based estimator of the variance. To be specific we get an estimator of the variances by dividing by $\sqrt{2}$ the median absolute deviation of disjoint successive differences. This is common practice in the change-point literature (see for example \citep{Fryzlewicz2012}). Such estimators are less sensitive to any shift in the mean than the classical ones.
For the Gaussian kernel we then used $\delta=1$.

As mentioned earlier, one main asset of kernels is that they allow to easily perform data fusion, which consists here in combining several data profiles to increase the power of detecting small changes arising at the same location in several of them.
Here the joint segmentation of the two-dimensional signal (TCN, BAF) is carried out by defining a new kernel as the sum of two coordinate-wise kernels \cite{Aron:1950}, that is
\begin{align}
k(x_1,x_2) =  k(c_1,c_2) +  k(b_1,b_2)  \label{eq.def.noyau.mixture}
\end{align}
with $x_1 = (c_1,b_1)$ and $x_2=(c_2,b_2)$ where the first coordinates of $x_1$ and $x_2$ refer to TCN and the second ones to BAF.
	\begin{rk}
		Let us point out that many alternative ways exist to build such a ``joint kernel'', using the standard machinery of reproducing kernels exposed in \cite{Aron:1950,Gart_2008}. 
		
		For instance replacing the sum in Eq.~\eqref{eq.def.noyau.mixture} by a product of kernels is possible. With the Gaussian kernel, this amounts to consider one Gaussian kernel applied to a mixture of squared norms where each coordinate receives a different weight depending on its influence.
		Another promising direction is to exploit some available side information about the importance of each coordinate in detecting change-points. This can be done by considering a convex sum of kernels where the weights reflect this \emph{a priori} knowledge.
		
		Finally let us mention that designing the optimal kernel for a learning task is a widely open problem in the literature even if some attempts exist (see Section~7.2 in \cite{Arl_Cel_Har:2012:v1} for a thorough discussion, and \cite{GSSBPFS_2012} for a first partial answer with two-sample tests).
	\end{rk}

\subsubsection{ECP}
\label{sec:ecp-description}

The ECP procedure \citep{matteson2014nonparametric} (earlier discussed in Section~\ref{sec.link.other.approaches}) has been also introduced in our comparison since it allows us to detect changes in the distribution of multivariate observations.

We used the implementation provided by the authors in the R-package \cite{ecp}. We used the default parameters $\alpha=1$ and $\ell = 30$ (minimum length of any segment). 
Let us notice that, unlike our kernel-based procedures relying on efficiently minimizing a prescribed penalized criterion, \ECP chooses the number of segments by iteratively testing each new candidate change-point by means of a permutation test, which makes it highly time-consuming on large profiles (around 15 minutes per profiles for $n=5000$ compared to 5 seconds for \KernSegGauss).

\subsubsection{Recursive Binary Segmentation (RBS)}
\label{sec:rbs-description}

In the recent paper by \cite{Pierre-Jean08092014}, it has been shown that for a known number of change-points the Recursive Binary Segmentation (\RBS) \citep{gey08using} is a state-of-the-art change-point procedure for analyzing (TCN, BAF) profiles. 
\RBS is a two-step procedure. In a first step it uses the binary segmentation heuristic (described in Section~\ref{sec.binary.segmentation}) on the (TCN) or (TCN,BAF) profile. In a second step it uses dynamic programming 
on the set of changes identified by the binary segmentation heuristic. 
We refer interested readers to \cite{Pierre-Jean08092014}\xspace for a discussion as to why \RBS can outperform a pure dynamic programming strategy despite the fact that it provides only an approximation to the targeted optimization problem.

Since the present biological context is the same as that of  \cite{Pierre-Jean08092014}, we therefore decided to carry out the comparison between our kernel-based procedures and \RBS. 

From a computational perspective \RBS relies on the binary segmentation algorithm described in Algorithm~\ref{heur:v2}. 
The final segmentation output by \RBS is then an approximate solution to the optimization problem (in the same way as \aKern), while being efficiently computed as illustrated by Figure~\ref{fig:time2}.


\subsection{Performance assessment}
\label{sec:perf-crit}

The quality of the resulting segmentations is quantified in two ways.
First we infer the ability of the procedure to provide a reliable estimate of the regression function by computing the quadratic risk of the estimator based on the TCN profile (Section~\ref{sec:piec-const-mean}). 
Second, we also assess the quality of the estimated segmentations by measuring the discrepancy between the true and estimated change-points using the Frobenius distance
 (Section~\ref{sec:accur-segm-1}).


\subsubsection{Risk of a segmentation}
\label{sec:piec-const-mean}

From a practical point of view, there is no hope to recover true change-points in regions where the signal-to-noise ratio is too low without including false positives, which we would like to avoid.
In such non-asymptotic settings, the quality of the estimated segmentation $\tau$ can be measured by the risk $R(\hat f^{\tau})$ which measures the gap between the regression function $f=(f_1,\ldots,f_n)\in\R^n$ and its piecewise-constant estimator based on $\tau$, that is $\hat f^{\tau}=\paren{\hat f^{\tau}_1,\ldots,\hat f^{\tau}_n}\in\R^n$.
This risk is defined by
\begin{align*}
R(\hat f^{\tau}) = \frac{1}{n}\sum_{i=1}^n \E\croch{ \paren{ f_i - \hat f^{\tau}_i}^2 } .
\end{align*}
In the following simulation results, the risks of all segmentations are always computed with respect to the regression function of the corresponding TCN profile.

\subsubsection{Frobenius distance}
\label{sec:accur-segm-1}

We also quantify the gap between a segmentation $\tau$ and the true segmentation $\tau^*$ by using the Frobenius distance \cite{LajugieBachArlot_2014} between matrices as follows.
First, for any segmentation $\tau=\paren{\tau_1,\tau_2,\ldots,\tau_{D} }$, let us introduce a matrix $M^\tau = \acc{M^\tau_{i,j}}_{1\leq i,j\leq n}$ such that
\begin{align*}
 M^\tau_{i,j} = \sum_{k=1}^{D} \frac{\1_{( \tau_{k}\leq i,j < \tau_{k+1} )} }{\tau_{k+1}-\tau_{k}}, \hspace*{2cm} \mbox{(with $\tau_1=1$ and $\tau_{D+1}=n+1$ by convention)}
\end{align*}
where $\1_{( \tau_{k}\leq i,j < \tau_{k+1})}=1$, if $i,j \in [\tau_{k},\tau_{k+1}[\cap\mathbb{N}$, and 0 otherwise.
Note that $M^\tau_{i,j}\neq 0$ if and only if $i,j$ are in the same segment of $\tau$, which leads to a block-diagonal matrix with D blocks (whose squared Frobenius norm is equal to D).
The idea behind the value in each block of this matrix is to define a one-to-one mapping between the set of segmentations in D segments and matrices whose squared Frobenius norm is D.

 

Let us now consider the matrix $M^{\tau^\star}$ defined from the true segmentation $\tau^*$ in the same way.
Then, the Frobenius distance between segmentations $\tau$ and $\tau^*$ is given, through the distance between matrices $M^\tau$ and $M^{\tau^\star}$, by 
\begin{align*}
d_F\paren{\tau,\tau^\star} = \norm{ M^\tau - M^{\tau^\star}}_F=  \sqrt{ \sum_{i,j=1}^n  \paren{ M^\tau_{i,j} -M^{\tau^\star}_{i,j}  }^2} .
\end{align*}

\subsection{Results}

In our experiments, we successively considered two types of signals: $(i)$ the total copy number profiles (TCN)
and $(ii)$ the joint profiles in $\R^2$ made of (TCN,BAF).

\subsubsection{Comparison with \KernSegLin and \ECP for a high tumor percentage (easy case)}
First we compare all approaches in the simple case where the tumor percentage is equal to 100\%.
The performances, using only the TCN or the (TCN,BAF) profiles, are reported in Figure~\ref{fig:easy_case} and measured in terms of accuracy (Left) and risk (Right).
\begin{figure}[h!]
	\hspace*{-1cm}
	\begin{tabular}{cc}
	\includegraphics[width=.5\textwidth, clip=true, trim= 0cm 0cm 1cm 2cm]{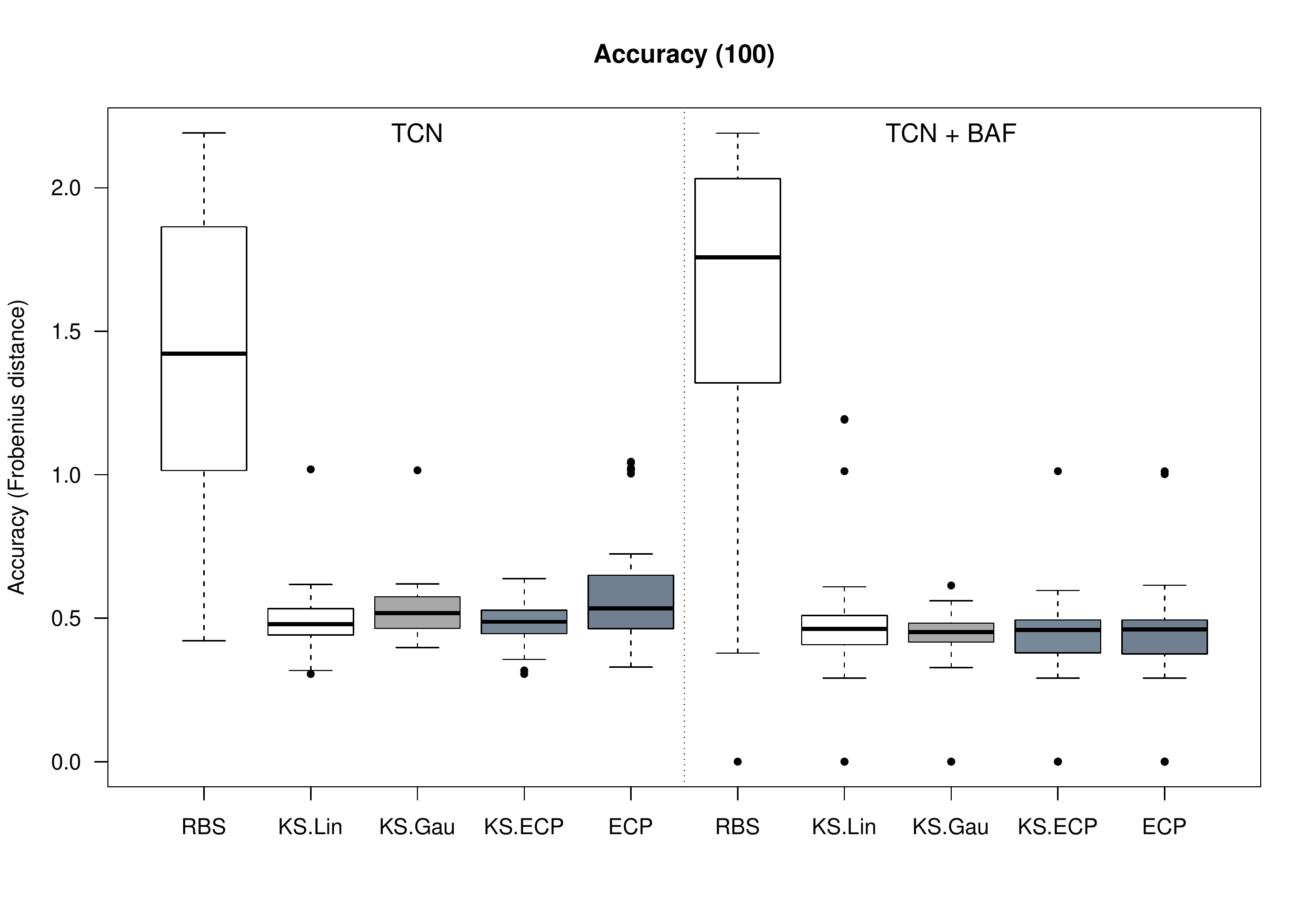} &
        \includegraphics[width=.5\textwidth, clip=true, trim= 0cm 0cm 1cm 2cm]{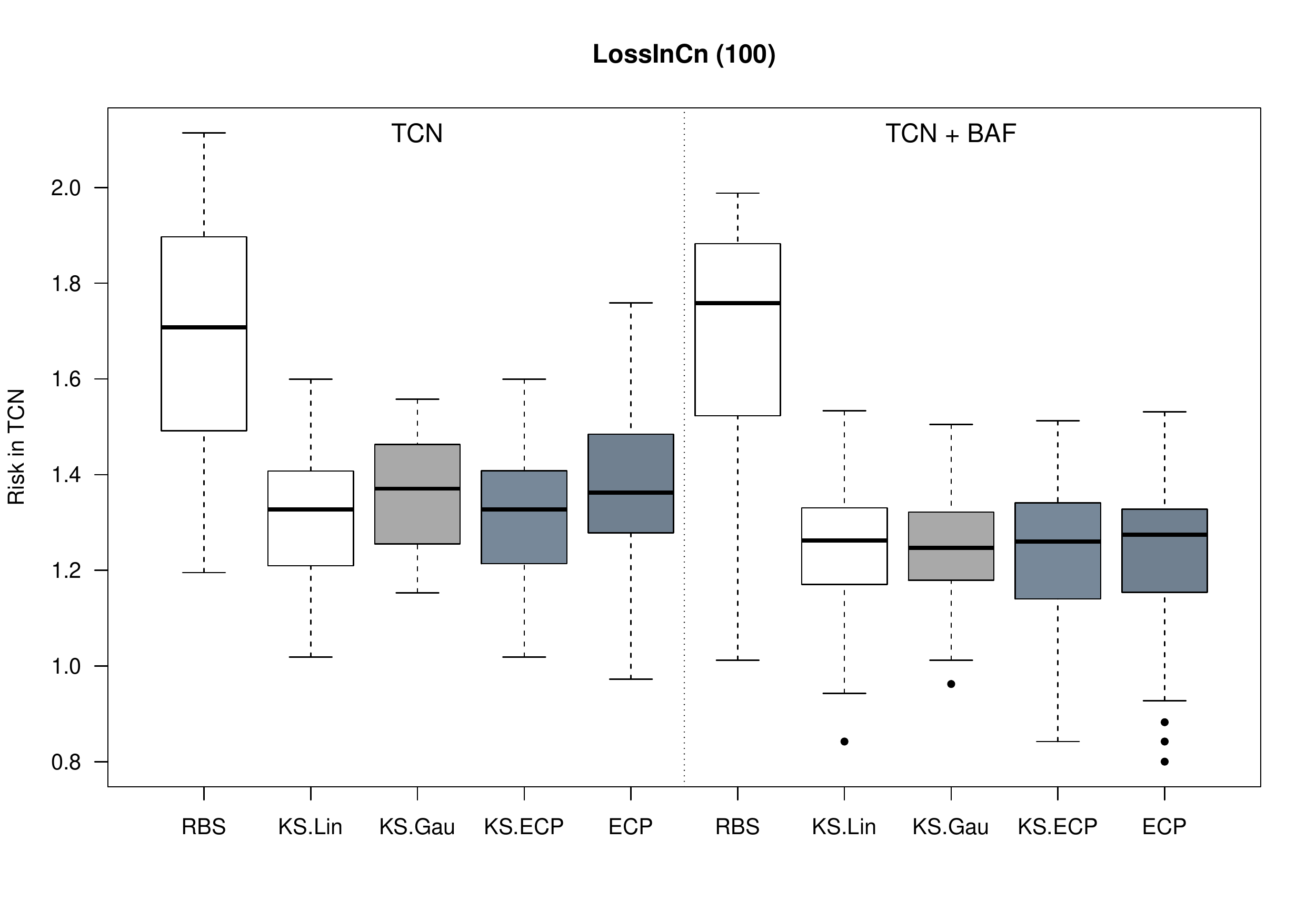}
        \end{tabular}
	\caption{Accuracy (left) and Risk (right) on TCN and (TCN,BAF) profiles with a tumor percentage of 100\%. Boxplots of \RBS, \ECP, \KernSegLin, \KernSegGauss and \KernSegECP for their selected number of change-points ($\hat{D}$) are shown.}
		\label{fig:easy_case}	
\end{figure}

In all these experiments \RBS clearly performs badly. We believe this is mostly due to the poor estimation of the number of segments made by \RBS. 
Indeed the performances of \RBS
are closer to the ones of other approaches when considering the true number of segments (results not shown here).

We then compare \KernSegLin to \KernSegGauss, \KernSegECP,  and \ECP.
With TCN data, \KernSegLin has a small advantage over \KernSegGauss (with an average accuracy difference of 0.03 and a p-value of 0.012) and \ECP (with an average accuracy difference of 0.1 and p-values of 0.007).
This is also true when considering the risk.
\KernSegECP has a slightly better empirical average accuracy than \KernSegLin but this difference is not significant.
Let us also mention that none of the differences are found significant with the (TCN, BAF) profiles.
It is our opinion that in this simple scenario all true change-points arise mostly in the mean of the distribution. 
Thus it is noticeable that the performances of approaches also looking for changes in the whole distribution (like \ECP, \KernSegGauss, \KernSegECP) are (almost) on par with those specifically looking for change-points in the mean (like \KernSegLin and \RBS for the true number of change-points $D^\star$).

We also compared \ECP to \KernSegGauss and \KernSegECP. We found no differences except between \ECP and \KernSegECP for TCN profiles. In that case \KernSegECP has significantly better accuracy and risk than \ECP. But this difference remains small as can be seen on Figure~\ref{fig:easy_case}.

Note that for all approaches, performances on (TCN,BAF) profiles are slightly better than those with TCN profiles (p-values smaller than $10^{-4}$).

\subsubsection{Constraint on the segment sizes for a low tumor percentage (difficult case)}
We then turn to the more difficult case where the tumor percentage is equal to 50\%.
In this scenario excluding segments with less than 30 points (as
is done by default in \ECP) is beneficial. 
Figure~\ref{fig:hard_minsize} illustrates this strong improvement when adding this constraint to \KernSegLin, \KernSegGauss and \KernSegECP and when considering the true number of change-points $D^\star=10$ (p-values of respectively ($8.10^{-9}$, $9.10^{-3}$ and $10^{-4}$).
More generally it is our experience that such a constraint can greatly improve performances when the signal-to-noise ratio is low.
For this reason, in the remainder of our experiments and for a tumor percentage of 50\%, we will report results including the constraint on the segment sizes ($\ell =30$). 
Let us also mention that for higher tumor percentages adding the constraint does not change the segmentation in $D^\star$ segments recovered by \KernSegLin, \KernSegGauss and \KernSegECP.

\begin{figure}[t]
	\centering
	\includegraphics[width=.8\textwidth, clip=true, trim= 0cm 0cm 1cm 2cm]{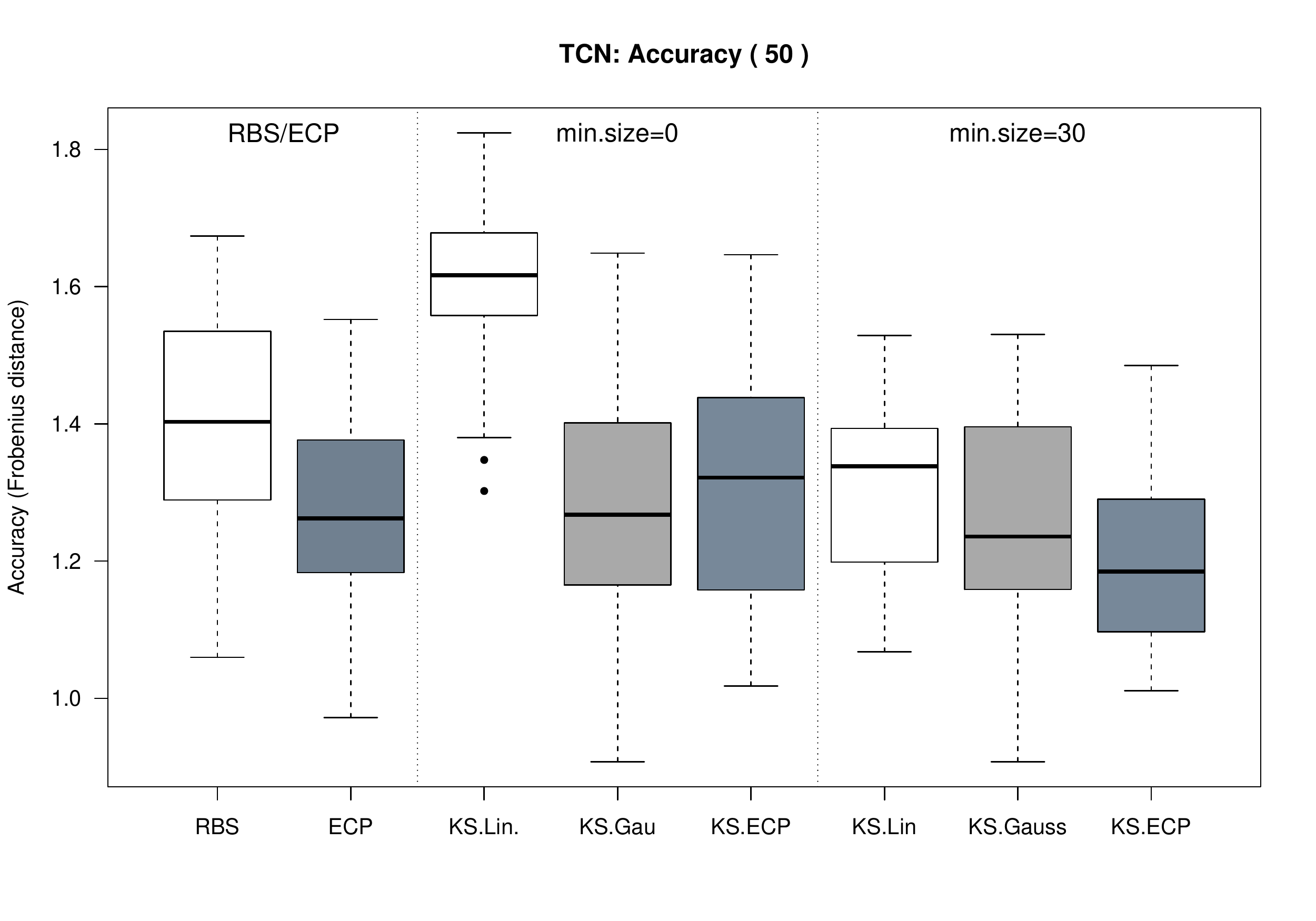}
	\caption{Performances of \KernSegLin, \KernSegGauss and \KernSegECP for the true number of changepoints
$D^\star =10$ with or without a constraint on the minimal size of segments (left: $\ell \geq 1$, right: $\ell \geq 30$). 
The results for \RBS and
\ECP for $D^\star =10$ are also reported. \RBS does not include a constraint while \ECP has a default minimal size of $30$.
}
	\label{fig:hard_minsize}. 
\end{figure}

\subsubsection{Comparison with \KernSegLin and \ECP for a low tumor percentage (difficult case)}

We compared \KernSegLin to \KernSegGauss, \KernSegECP,  and \ECP for a tumor percentage of 50\%.
The accuracy of all these approaches is reported in Figure~\ref{fig:hard_case}. The minimum length of any segment is fixed at $\ell=30$ for all approaches (except \RBS as it is not possible) and the number of segments is estimated.
\begin{figure}[h!]
	\hspace*{-1cm}
	\begin{tabular}{cc}
	\includegraphics[width=.5\textwidth, clip=true, trim= 0cm 0cm 1cm 2cm]{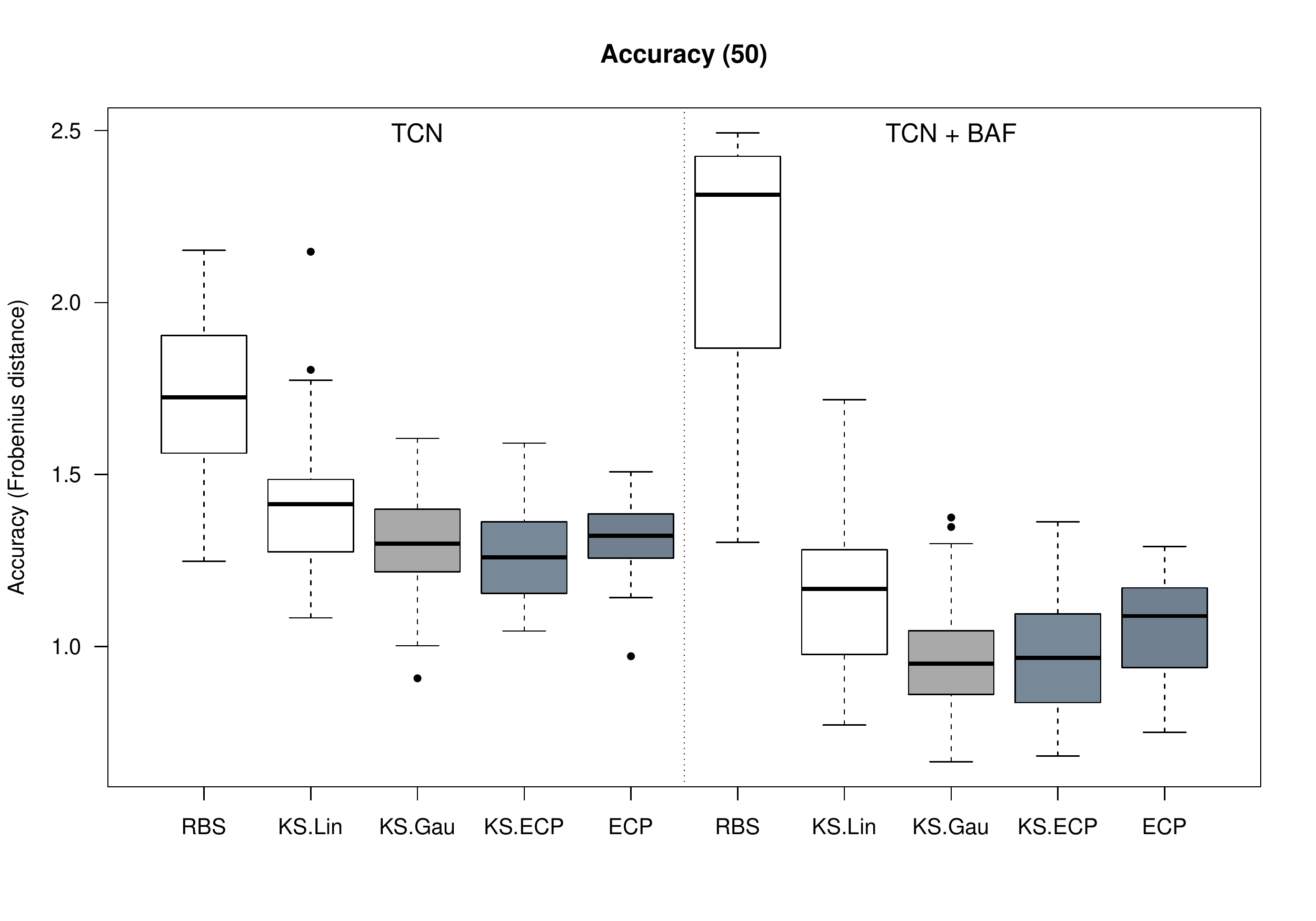} &
        \includegraphics[width=.5\textwidth, clip=true, trim= 0cm 0cm 1cm 2cm]{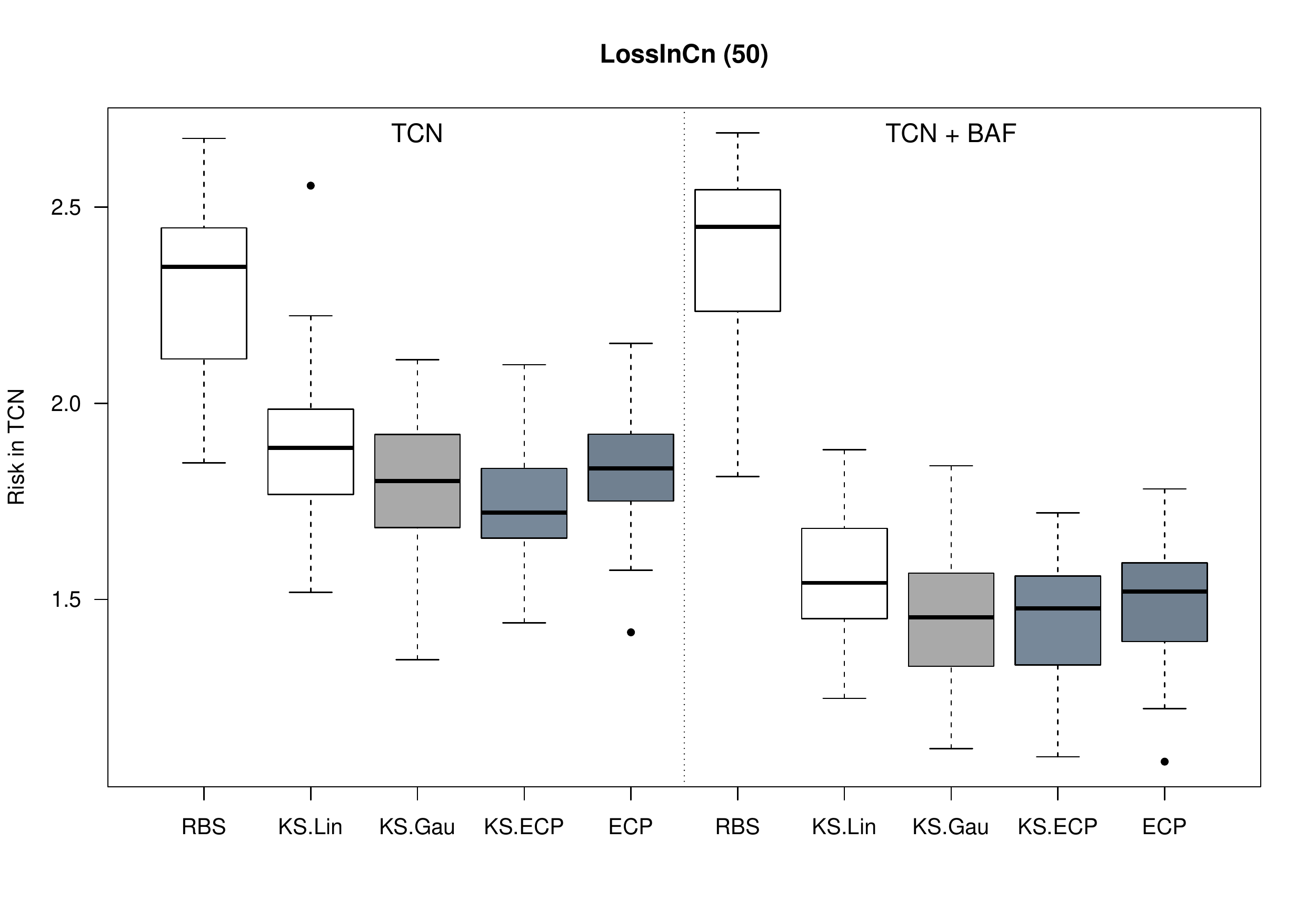}
        \end{tabular}
	\caption{Accuracy (left) and Risk (right) on TCN and (TCN,BAF) profiles with a tumor percentage of 50\%. Boxplots of \RBS, \ECP, \KernSegLin, \KernSegGauss and \KernSegECP for their selected number of changepoints ($\hat{D}$) are shown.}	
	\label{fig:hard_case}
\end{figure}

In all these experiments \RBS performs badly. We believe this is because it poorly selects the number
of segments and also because it does not include a constraint on segment sizes.

We compared \KernSegLin to \KernSegGauss, \KernSegECP,  and \ECP.
For both TCN and (TCN,BAF) profiles \KernSegLin performs worse than \KernSegGauss, \KernSegECP,  and \ECP in terms of accuracy and risk (all p-values are smaller than $2.10^{-4}$).
In this more difficult scenario, changes do not arise only in the mean of the distribution, which gives an advantage to approaches looking for changes in the whole distribution and not only in the mean as \KernSegLin does.

We then compare \ECP to \KernSegGauss and \KernSegECP. 
First, \KernSegGauss seems to have a slightly better accuracy and risk than \ECP for both TCN and (TCN,BAF) profiles. 
Two of these differences are found significant with a cut-off of 5\% and none with a cut-off of 1\%. 
This leads us to conclude that \ECP and \KernSegGauss have similar performances in the present experiments.
Second, \KernSegECP has a slightly better accuracy and risk than \ECP in TCN for both TCN and (TCN,BAF) profiles. 
All of these differences are found significant (for the accuracy in (TCN,BAF) $p=0.0076$, in TCN $p=0.015$, for the risk in (TCN,BAF) $p=0.0081$ and in TCN $p=0.00016$). 
Although significant these differences remain small (about
three times smaller than the differences between \KernSegLin and \ECP).

Finally it should be noted that \KernSegECP is faster than \ECP for a profile of $n=5000$ (5 seconds against 15 minutes).
All of this leads to conclude that, in our simulation experiments, \KernSegGauss and \KernSegECP are the best change-point detection procedures among the considered ones since they perform as well as \ECP while being by far less memory and time consuming.

\subsubsection{Estimation of the number of segments including the minimum length constraint}

Let us now assess the behaviour of the model selection procedure derived in Section~\ref{sec.constrained.segmentations} by taking into account the new constraint on the minimal length of the candidate segments.
\begin{figure} 
\centering
\includegraphics[width=.7\textwidth, clip=true, trim= 0cm 0cm 1cm 2cm]{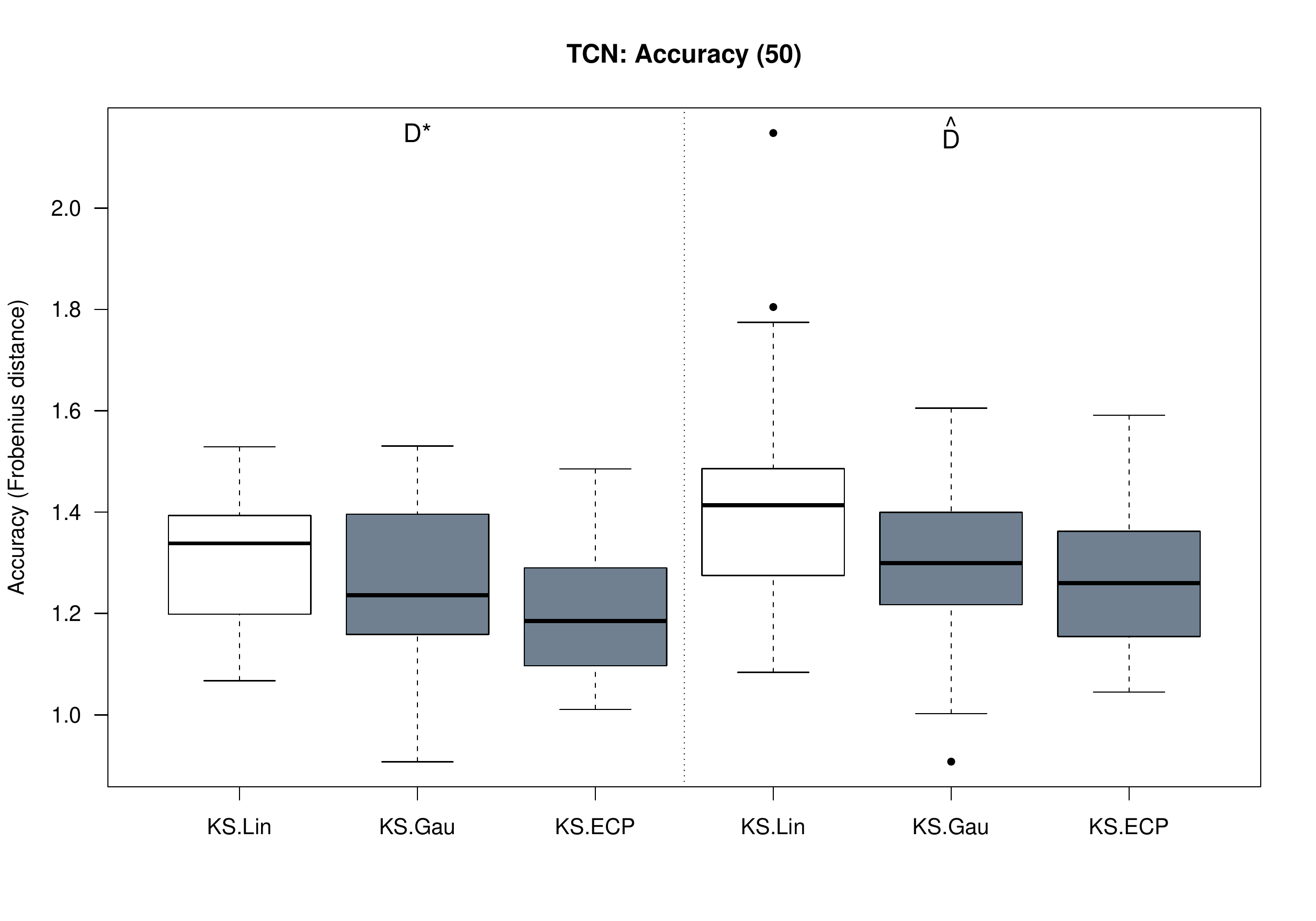} 
\caption{Accuracy of \KernSegLin,  \KernSegGauss and \KernSegECP on (TCN,BAF) for a tumor percentage of 50\% for $D*$ (left) and $\hat{D}$ (right). }
\label{fig:model_sel}
\end{figure}

From Figure~\ref{fig:model_sel} it can be seen that whatever the kernel the performances of the \Kern procedure at the estimated number of segments are worse than those at $D^\star$. But this difference remains very small.
This empirically validates the use of our modified penalty taking into account a constraint on the size of the segments.
We recall that adding this constraint is important in low-signal-to-noise settings, which are common in practice.

\subsubsection{Quality of the approximation}
\label{sec:approx}

The purpose of the present section is to illustrate the behaviour of \aKern (in terms of statistical precision) as an alternative to \Kern (which is more time consuming).
Since we do not provide any theoretical warranty on the model selection performances of \aKern, we only show its results for several values of $p\in \acc{4,10,40,80,160}$ at the true number of segments $D^\star $.
For each value of $p$ and each of the TCN and BAF profiles, we compute the approximation by: $(i)$ evaluating the smallest and largest observed value (respectively denoted by $m$ and $M$), $(ii)$ using an equally spaced grid of $p$ deterministic values between $m$ and $M$ and $(iii)$ use those $p$ values to perform the approximation of the Gram matrix.

From Figure~\ref{fig:approx.accuracy} it clearly appears that the number of points used to build the low-rank approximation to the Gram matrix is an influential parameter that has to be carefully fixed.
However as long as $p$ is chosen large enough, the approximation seems to provide very similar results.
\begin{figure}[t]
	\centering
	\includegraphics[width=.7\textwidth, clip=true, trim= 0cm 0cm 1cm 2cm]{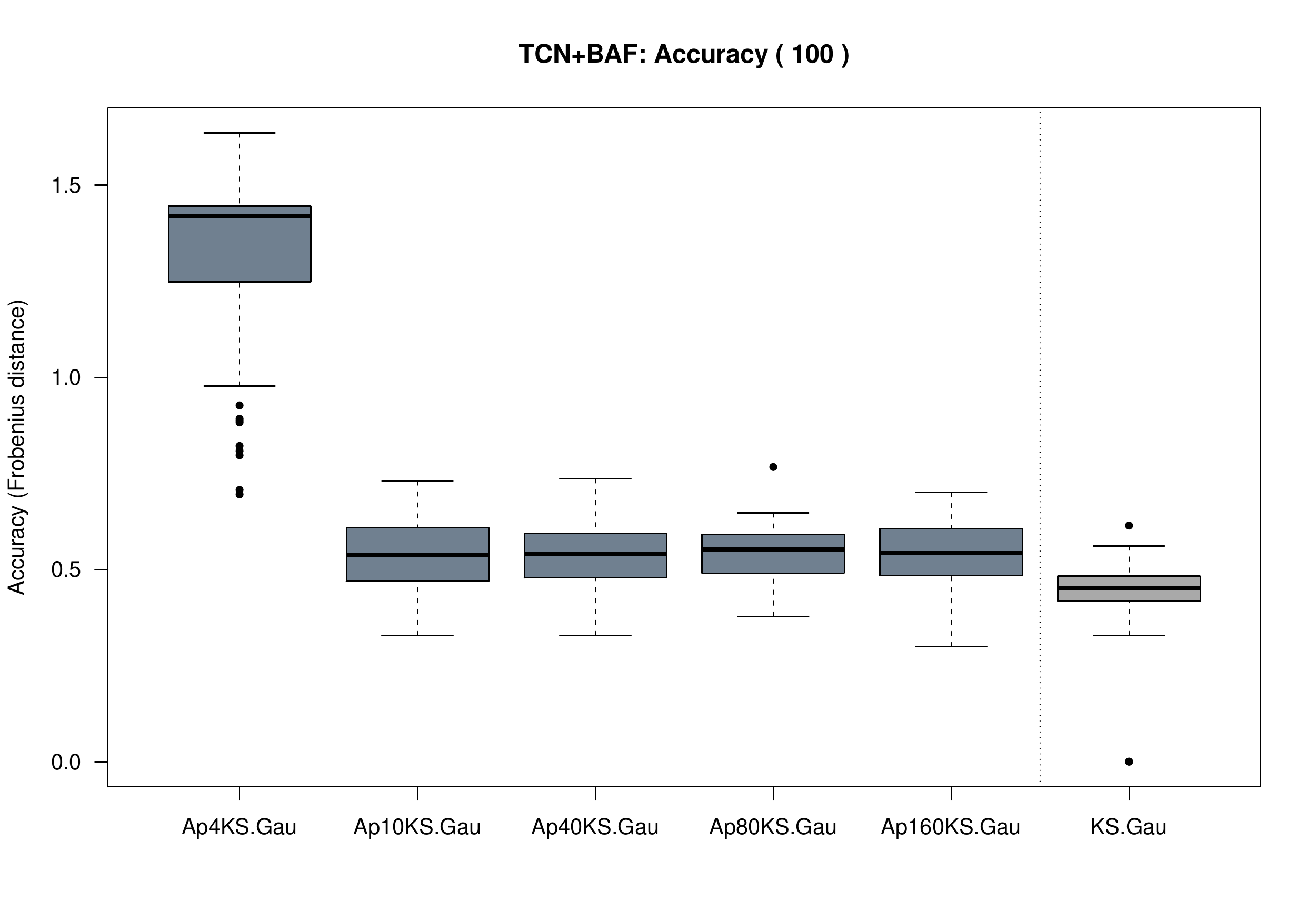}
	\caption{ Accuracy of \aKern with the Gaussian Kernel and for various $p$ and of \KernSegGauss on (TCN,BAF) for a tumor percentage of 50\% for $D*$.  } 
	\label{fig:approx.accuracy}
\end{figure}
This suggests that one should find a trade-off between the statistical performances and the computation cost. Indeed from a statistical point of view increasing $p$ is beneficial (or at least not detrimental). In contrast from a computationnal point of view increasing $p$ is detrimental and increases the complexity in time ($O(p^2n)$).

Let us finally emphasize that for large enough $p$ the performances of \aKern are very close to those of \KernSegGauss.
Given the low time complexity of \aKern compared to \KernSegGauss we argue that for large profiles ($n \gg 10^5$)
\aKern could be an interesting alternative to \KernSegGauss.

Nevertheless several questions related to the use of \aKern remain open. 
For instance, the optimal way to build the low-rank approximation of the Gram matrix is a challenging question which can be embedded in the more general problem of choosing the optimal kernel.
Designing a theoretically grounded penalized criterion to perform model selection with \aKern is also a crucial problem which remains to be addressed.

\section{Conclusion}
\label{sec:discussion}

Existing nonparametric change-point detection procedures such as that of \cite{Arlot2011} exhibit promising statistical performances. Yet their high computational costs (time and memory) are severe limitations that often make it difficult to use for practitioners. Therefore an important task is to develop computationally efficient algorithms (leading to exact or approximate solutions) reducing the time and memory costs of these statistically effective procedures. 

In this paper we focus on the multiple change-points detection framework with reproducing kernels.
We have detailed a versatile (\emph{i.e.} applicable to any kernel) exact algorithm which is quadratic in time and linear in space. We also provided a versatile approximation algorithm which is linear both in time and space and allows to deal with very large signals ($n\geq 10^6$) on a standard laptop. 
The computational efficiency in time and space of these two new algorithms has been illustrated on empirical simulation experiments showing that the new algorithms is more efficient than its direct competitor \ECP. 
The statistical accuracy of our kernel-based procedures has been empirically assessed in the setting of DNA copy numbers and allele B fraction profiles. 
In particular, results illustrate that characteristic kernels (enabling the detection of changes in any moment of the distribution) can lead to better performances than procedures dedicated to detecting changes arising only in the mean.

\section*{Acknowledgements}

This work has been funded by the French Agence Nationale de la Recherche (ANR) under reference ANR-11-BS01-0010, by the CNRS under the PEPS BeFast, by the CPER Nord-Pas de Calais/FEDER DATA Advanced data science and technologies 2015-2020, and by Chaire d'excellence 2011-2015 Inria/Lille 2.

\section*{References}
\label{sec:references}

\bibliographystyle{elsarticle-harv}
\bibliography{biblio,biblio_Alain}

\end{document}